\documentclass[10pt]{article}
\font\smallit=cmti10

\usepackage{amssymb,latexsym,amsmath,epsfig,amsthm} 
\usepackage{chessboard}
\usepackage[hidelinks]{hyperref}
\usepackage{tikz} 
\usetikzlibrary{arrows}
\setchessboard{margin=true,
               moverstyle=triangle,
               tinyboard,
               labelbottom=true,
               labelleft=true,
               labelfontsize=6pt,
               labelleftwidth=2ex,
               linewidth=.05ex,
               border=false
              }

\makeatletter

\renewcommand\section{\@startsection {section}{1}{\z@}
{-30pt \@plus -1ex \@minus -.2ex}
{2.3ex \@plus.2ex}
{\normalfont\normalsize\bfseries}}

\renewcommand\subsection{\@startsection{subsection}{2}{\z@}
{-3.25ex\@plus -1ex \@minus -.2ex}
{1.5ex \@plus .2ex}
{\normalfont\normalsize\bfseries}}

\renewcommand{\@seccntformat}[1]{\csname the#1\endcsname. }

\makeatother

\newcommand{\Ch}{{\mathfrak{Ch}}}
\newcommand{\omegaoneCh}{\omega_1^\Ch}
\newcommand{\omegaoneChi}{\omega_1^{\baselineskip=0pt\vtop to 7pt{\hbox{$\scriptstyle\Ch$}\vskip-1pt\hbox{\,$\scriptscriptstyle\sim$}}}}
\newcommand{\omegaoneChc}{\omega_1^{\Ch,c}}
\newcommand{\omegaoneChthree}{\omega_1^{\Ch_3}}
\newcommand{\omegaoneChthreei}{\omega_1^{\baselineskip=0pt\vtop to 7pt{\hbox{$\scriptstyle\Ch_3$}\vskip-1.5pt\hbox{\,$\scriptscriptstyle\sim$}}}}
\newcommand{\omegaoneChthreec}{\omega_1^{\Ch_3,c}}
\newcommand{\Z}{{\mathbb Z}}
\newcommand{\df}{\it}
\newtheorem{theorem}{Theorem}
\newtheorem{definition}[theorem]{Definition}
\newtheorem{observation}[theorem]{Observation}
\newtheorem{conjecture}[theorem]{Conjecture}
\newtheorem{lemma}[theorem]{Lemma}
\newtheorem{corollary}[theorem]{Corollary}
\newcommand{\Godel}{G\"odel}
\def\<#1>{\langle#1\rangle}
\newcommand{\smalllt}{\mathrel{\mathchoice{\raise2pt\hbox{$\scriptstyle<$}}{\raise1pt\hbox{$\scriptstyle<$}}{\raise0pt\hbox{$\scriptscriptstyle<$}}{\scriptscriptstyle<}}}
\newcommand{\ltomega}{{{\smalllt}\omega}}
\newcommand{\of}{\subseteq}
\newcommand{\rank}{\mathop{\rm rank}}
\newcommand{\set}[1]{\{\,{#1}\,\}}

\begin{document}

\begin{center}
\uppercase{\bf Transfinite game values in infinite chess}
\vskip 20pt
{\bf C. D. A. Evans\footnote{The first author holds the chess title of National Master.}}\\
{\smallit Program in Philosophy\\
   The Graduate Center of The City University of New York\\
   365 Fifth Avenue, New York, NY 10016}\\
{\tt c.alexander.evans@gmail.com}\\
\vskip 10pt
{\bf Joel David Hamkins\footnote{The research 
        of the second author has been supported in part
        by NSF grant DMS-0800762, PSC-CUNY grant 64732-00-42 and Simons
        Foundation grant 209252. The second author wishes to thank Horatio M. Hamkins for assistance with chess analysis. The authors are grateful
        for a key suggestion made by W. Hugh Woodin to the second author at the Computational Prospects of Infinity Workshops at the National University of Singapore in August, 2011 and for helpful remarks of the anonymous referee, including especially a simplification in the position of figure~\ref{Figure.ValueOmega}. Commentary concerning this paper can be made at http://jdh.hamkins.org/game-values-in-infinite-chess, where animations of some of the positions can also be found.}}\\
{\smallit Program in Mathematics, Program in Philosophy, Program in Computer Science\\
 The Graduate Center of The City University of New York\\
 365 Fifth Avenue, New York, NY 10016\\ 
 \&\quad Mathematics, College of Staten Island of CUNY, Staten Island, NY 10314}\\
 {\tt jhamkins@gc.cuny.edu, http://jdh.hamkins.org}\\
\end{center}
\vskip 30pt
\phantom{\centerline{\smallit Received: , Revised: , Accepted: , Published: }} 
\vskip 30pt

\centerline{\bf Abstract}
\noindent
 We investigate the transfinite game values arising in infinite chess, providing both upper and lower bounds on the supremum of these values---the omega one of chess---denoted by $\omegaoneCh$ in the context of finite positions only and by $\omegaoneChi$ in the context of all positions, including those with infinitely many pieces. For lower bounds to $\omegaoneChi$, we present specific positions with transfinite game values of $\omega$, $\omega^2$, $\omega^2\cdot k$ and $\omega^3$. By embedding trees into chess, we show that there is a computable infinite chess position that is a win for white if the players are required to play according to a deterministic computable strategy, but which is a draw without that restriction. Finally, we prove that every countable ordinal arises as the game value of a position in infinite three-dimensional chess, and consequently the omega one of infinite three-dimensional chess is as large as it can be, namely, $\omegaoneChthreei=\omega_1$.

\pagestyle{myheadings}
\thispagestyle{empty}
\baselineskip=12.875pt
\vskip 30pt

%
%
%
%
In infinite chess---chess played on an infinite chessboard---the familiar chess pieces move about on an expansive edgeless plane tiled with black and white squares in the usual chessboard pattern, stretching indefinitely in all four directions. Each player aspires to place the opponent's king into checkmate. There is no standard starting position, but rather one investigates the nature of the game beginning from a specified initial position, not necessarily finite. Since checkmates, when they occur, do so after finitely many moves, chess is what is known technically as an open game and consequently is subject to the game-theoretic analysis of open games using ordinal game values. In particular, the winning positions in infinite chess are exactly those having an ordinal game value, and the winning strategies are those that eventually reduce this game value, forcing it zero, which is when checkmate occurs.

In this article, we investigate the range of transfinite game values exhibited by positions of infinite chess. We shall exhibit specific positions in infinite chess having transfinite game values of $\omega$, $\omega^2$, $\omega^2\cdot k$ and $\omega^3$. We define the omega one of chess to be the supremum of the game values exhibited by positions in infinite chess, using this phrase to refer both to the lightface form $\omegaoneCh$, which is the supremum of the game values arising in positions having only at most finitely many pieces, and the boldface form $\omegaoneChi$, which is the supremum of the game values arising in any position of infinite chess, including those having infinitely many pieces. We conjecture that the soft theoretical upper bounds provided by theorem \ref{Theorem.UpperBounds} are optimal, although the best-known lower bounds, including the specific positions provided in this article, are not yet close. By embedding infinite finitely branching trees into positions of infinite chess, we show in section \ref{Section.ComputableStrategies} that there is a computable infinite chess position, such that our judgement of whether it is a win for white or a draw depends on whether we insist that the players play according to a deterministic computable procedure. Namely, in the category of computable play the position is a win for white in the sense that white has a computable strategy that defeats any computable strategy for black; but without that computable restriction the position is a draw, since black has a non-computable strategy forcing a draw. Finally, by embedding infinitely-branching trees into positions of infinite three-dimensional chess, explained in section \ref{Section.ThreeD}, we prove in theorem \ref{Theorem.3DAnyValueAtttained} the analogue of our conjecture \ref{Conjecture.OmegaOneCH=OmegaOne} for infinite three-dimensional chess, by proving that every countable ordinal arises as the game value of a position in infinite three-dimensional chess, and consequently the omega one of infinite three-dimensional chess is as large as it can be, namely, $\omegaoneChthreei=\omega_1$.

Before proceeding, let us say a bit more about the rules of infinite chess. The game has two players, black and white, each controlling pieces of their own color, taking turns. The chess board consists of squares arranged in the structure of the integer plane $\Z\times\Z$, each of which is either empty or holds at most one chess piece. The chess pieces---kings, queens, rooks, bishops, knights and pawns---move on the infinite board in the manner that any chess player would expect. Because infinite chess often involves extremely long play and we are accordingly patient, we abandon the standard tournament 50-move rule (stating that the game is a draw, if there should be 50 moves without a capture or a pawn move). As there are no edges to the board, the pawn promotion rules never take effect, and there is similarly no opportunity for {\it en passant} or castling. The game ends when one player places the other's king into checkmate, meaning that the king is attacked in such a way that cannot be prevented. Games for which play continues unendingly are a draw; neither player has achieved checkmate and so neither player has won. Because of this, we may abandon the three-fold-repetition draw rule, because repeating indefinitely leads to infinite play, and finite repetitions needn't have been repeated in the first place, if a draw was not desired. If a player has no legal move on their turn, then the position is {\it stalemate}, which is a draw. It is customary in infinite chess to consider positions in which each side has at most one king. In a position where a player has no king, the other player hopes for at best a draw, since there is no possibility of checkmate. In all the figures of this article, white pawns move upward and black pawns downward, and one should presume that the board continues beyond what is pictured, extending whatever pattern has been established in the visible portion of the board. Many of the diagrams include a turn indicator, a small black or white triangle, signalling which player is to move first.

Let us introduce some terminology that will assist our later discussions. A {\df board position} in infinite chess is an assignment to each location on the board a value indicating whether or not there is a piece there and if so, which piece. A {\df position} is a board position together with an indication of whose turn it is to play. A position is said to be {\df computable} if the board position function is a Turing computable function. For decidability questions, we insist on a representation that explicitly indicates of each piece type and the empty-square type how many squares on the board exhibit that type, whether this is a finite natural number or infinity. In particular, one should be able to compute directly from a chess position representation exactly how many pieces there are on the board. For chess positions having only finitely many pieces, it follows that an equivalent representation would simply be a complete list of the pieces that occur and their locations. A {\df strategy} for one of the players from a given position $p$ is a function that tells the player, given the prior sequence of moves from $p$, what he or she is to play next. Once the initial position $p$ has been fixed, then a strategy $\tau$ from $p$ is a function from finite sequences of moves to a next move, and is thus naturally coded by a real. A strategy $\tau$ is {\df computable}, if this function is a computable function in the sense of Turing computability. If $\tau$ and $\sigma$ are strategies, then $\tau*\sigma$ is the play resulting from having the first player play strategy $\tau$ and the second player play strategy $\sigma$. The strategy $\tau$ defeats $\sigma$, if this play results in a win for the first player, and similarly for a loss or a draw. A strategy $\tau$ is a {\df winning} strategy, if it defeats all opposing strategies.

The topic of infinite chess has recently been considered in several popular math forums, featuring in several MathOverflow questions, including Johan W\"astlund's question \cite{MO63423Wästlund:CheckmateInOmegaMoves}, which directly inspired this article, as well as in Richard Stanley's question \cite{MO27967Stanley:DecidabilityOfInfiniteChess}, which had inspiried the second author's previous joint article \cite{BrumleveHamkinsSchlicht2012:TheMateInNProblemOfInfiniteChessIsDecidable}, and also in the IBM Research Ponder This Challenge for December 2011 \cite{IBMPonderThisChallenge2011:InfiniteChess}.

\section{Game values and the omega one of chess}\label{Section.TheOmegaOneOfChess}

Chess players are familiar with the rich collection of chess problems falling under the descriptions, {\it mate-in-$1$}, {\it mate-in-$2$}, and so on. A position is said to be {\df mate-in-$n$} for the first player, if there is a strategy leading to checkmate in at most $n$ moves, regardless of how the opponent plays. This familiar mate-in-$n$ concept is generalized in the context of infinitary game theory by the concept of  ordinal game values, a concept that is applicable not only to infinite chess, but to any open game, a game which when won, is won at a finite stage.\footnote{The ordinal game value of a game is also commonly known as the ordinal {\df rank} of the game, but we will not use this term this way in this article, since {\df rank} in chess already refers to a horizontal row on the board. Although the possibility of a draw in chess introduces a complication to the standard treatments of open games, nevertheless the theory
of ordinal game value works fine.} The basic idea is to assign ordinal values to the positions from which white can force a win, and these values provide a measure of white's distance from victory. A mate-in-$n$ position for white, for example, where $n$ is a finite natural number and where $n$ is optimal, will have value exactly $n$. A position with value $\omega$, the first infinite ordinal, in contrast, will be a position with black to play, but every play results in a mate-in-$n$ position for white, with $n$ as large a finite number as desired. Although the value is infinite, white can play from such a position so as to win in finitely many moves; but the number of moves is controlled by black, determined by his first move. For larger game values, white will still force a win in finitely many moves, but black will exert control not only at the beginning but also periodically during play, in a way we shall explain, determining how much longer play will continue. The general definition of game value is as follows.

\begin{definition}\rm
 The {\df game value} (for white) of a position in infinite chess is defined as for any open game by recursion. The positions with value $0$ are precisely those in which white has already won.\footnote{Although they will cause no trouble in this article, one should be aware that when one allows truly arbitrary starting positions in infinite chess, then there are various weird boundary cases to consider, involving positions that could not be reached from any other position with legal play, such as a position in which both kings are in check. We prefer to resolve them as discussed in \cite{BrumleveHamkinsSchlicht2012:TheMateInNProblemOfInfiniteChessIsDecidable}, following the idea that checkmate has to do fundamentally with ensuring the necessary capture of the opposing king.} If a position $p$ has white to move, then the value of $p$ is $\alpha+1$ if and only if $\alpha$ is minimal such that white may legally move from $p$ to a position with value $\alpha$. If a position $p$ has black to play, where black has a legal move from $p$, and every move by black from $p$ has a value, then the value of $p$ is the supremum of these values.
\end{definition}

The definition identifies the positions with ordinal value $\alpha$, by induction on $\alpha$; some positions may be left without any value. The {\bf fundamental observation of game values} is that the positions having an ordinal value are precisely the positions that are winning for white. Namely, if a position has an ordinal value, then white may play the {\it value-reducing} strategy, playing on each move so as to reduce the value; since black cannot move from a position with value to one with higher value or without value, it follows that white's value-reducing strategy leads to a strictly descending sequence of values, which because there is no infinite descending sequence of ordinals must eventually terminate at $0$, a win for white. Conversely, if a position has no ordinal game value, then black may follow the {\it value-maintaining} strategy, playing so as to maintain that the position has undefined value; from such a position, white may not play to a position having a value, since if white could play to a valued position, then the original position would have had a value. In this way, black can guarantee that play from an unvalued position either continues indefinitely or reaches a stalemate or a win for black, and thus attains at least a draw. So the positions having an ordinal value for white are precisely the positions from which white can force a win. Of course one may similarly define values for black, and the positions having a value for black are precisely the positions from which black may force a win.

Let us illustrate the game-theoretic meaning of these game values in the case of various smallish ordinals. It is not difficult to see that the positions with finite value $n$ are precisely the mate-in-$n$ positions for white, with optimal $n$. A position with value $\omega$, as we mentioned earlier, is a definite win for white, in finitely many moves, but the number of such moves can be pushed as high as desired by black. Facing a position with value $\omega$, black must make a move with a strictly smaller value, and by playing a move with value $n$ black in effect announces that he can ensure that the game will last another $n$ steps. Thus, a position with value $\omega$ (or higher) is one for which black can make an announcement at the beginning, which is a lower bound for the length of subsequent play before white wins.

A position with value $\omega+n$ is a position in which white can play so as to achieve a position with value $\omega$ in at most $n$ moves. A position with value $\omega\cdot 2$ is a position that is a win for white in finitely many moves, but in which twice during play, black may announce a number such that play will proceed at least that much longer before the next announcement; that is, on the first move, black announces the number $n$ by playing to a position with value $\omega+n$, from which he can force play to continue for $n$ additional moves before the position comes to have value $\omega$, at which time black announces a second number $m$ by playing to a position with value $m$, after which he can survive for another $m$ moves.

These examples show that from a position with value $\alpha$, the players are in effect counting down from $\alpha$. To make this idea precise, consider for any ordinal $\alpha$ the simple game {\it counting-down-from-$\alpha$}, in which black aims to count down from $\alpha$ while white observes. Specifically, black will play successively a strictly descending sequence of ordinals $\alpha>\alpha_0>\alpha_1>\alpha_2>\cdots$, while white says ``OK, fine'' at each step. The game continues until black plays $\alpha_n=0$,
at which time white wins. Since there is no infinite descending sequence of ordinals, white will inevitably win, but the length of the game is controlled by black. We find it to be a simple and informative exercise, which we encourage the reader to undertake, to show that the value of the counting-down-from-$\alpha$ game is exactly $\alpha$. In particular, faced with a successor ordinal $\beta+1$, black can attempt to conserve the value as much as possible, delaying his loss, by playing $\beta$, which maximizes the space of his remaining play. But faced with a limit ordinal $\beta$, black must play a strictly smaller ordinal, which in general will be considerably smaller than $\beta$, even if he may choose ordinals arbitrarily high in $\beta$. This is the essential feature of all games with a limit ordinal value $\beta$, for black can play to positions with values unbounded in $\beta$, but any such position will represent a loss in comparison with what might have been.\footnote{We should like to emphasize that the ordinal value $\alpha$ of an open game is not the same as it having Conway game number $\alpha$; the latter represents in a sense the number of free moves available for a designated player, in a context where one will consider sums of games. This is like the counting-down-from-$\alpha$ game where black counts down, while white plays elsewhere. The sum of two such games, with the roles swapped in one summand, is winning for the second player, who can copy moves, and the value is precisely $\alpha$, because the loser is in effect counting down from $\alpha$.}

In light of this, a position in infinite chess with value $\omega\cdot k$ is a position where black will be able to make $k$ announcements, each of which is the number of moves before the next announcement or the final loss when the announcements are used up. Each announced number $n_i$ corresponds to the ordinal $\omega\cdot i+n_i$ that black drops to from the limit ordinal $\omega\cdot(i+1)$, when the moves from his previous announcement were used up. For example, in the case of $\omega\cdot 3$, black might begin by announcing the number $547$, say, by playing to a position with value $\omega\cdot 2+547$, with subsequent play continuing for $547$ additional moves before arriving at a position with value $\omega\cdot 2$; at this time, black makes his second announcement, say, $2013$, by playing to a position with value $\omega+2013$, from which he can play for $2013$ additional moves before arriving at a position with value $\omega$; at this time, he makes his final announcement, say, $10^{2^{100!}}$, by playing to a position with that value; from this position he will be able to play for a considerable time, but ultimately lose at the end with no further announcements.

A position with value $\omega^2$ is a position that white will win in finitely many moves, but black can announce an arbitrary finite number $k$ at the start and play to a position with value at least $\omega\cdot k$. From that position, as in the previous paragraph, black will be able to make $k$ additional announcements, such that play will continue after each such announcement for at least that long before the next announcement. A position with value $\omega^2\cdot m$ is a position in which white can force a win, but black will be able to make $m$ large announcements, such that each large announcement will be a number, which is the number of additional announcements he will be enabled to make, each of which describes the length of play to be undertaken before the next announcement, and the next large announcement being made only after all the play from previous announcements is completed. In effect, black will play the counting-down-from $\omega^2$ game $m$ times in succession. A position with value $\omega^3$ is a position in which black can announce an arbitrary finite value $m$, and then proceed from a position with value $\omega^2\cdot m$ or in the interval $[\omega^2\cdot m,\omega^2\cdot(m+1))$.

Let us now begin to show that infinite values actually arise in infinite chess, even amongst positions with only finitely many pieces. In section \ref{Section.InfinitePositionsWithTransfiniteGameValue}, we will present infinite positions having larger transfinite values, and in section \ref{Section.ThreeD}, we will show that every countable ordinal arises as a game value of a position in infinite three-dimensional chess. We begin with minimal positions having arbitrarily large finite values.

\begin{observation}\label{Observation.MateInn}
 For each finite number $n$, there is a position in
 infinite chess, having four pieces, with value $n$. For $n> 5$, there are no such
 positions having three pieces, although there are three-piece positions with every value up to $5$.
\end{observation}

\begin{proof}
Consider the positions of the form in figure \ref{Figure.MateIn17}, shown in the case $n=17$, consisting of a white king, white queen, white rook and black king, with $n$ squares separating the kings.
\begin{figure}[h]
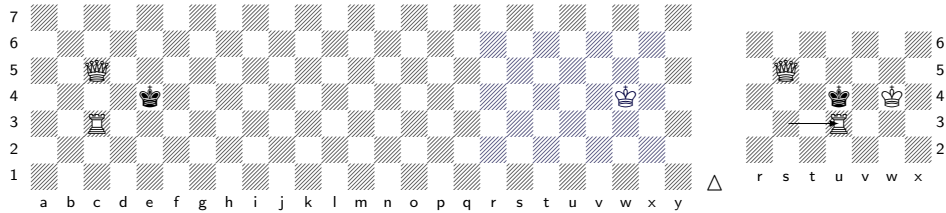

\chessboard[maxfield=y7,
            tinyboard,
            labelbottom=true,
            labelleft=true,
            labelfontsize=6pt, 
            labelleftwidth=1.6ex,
            mover=w,
            border=false,
            margin=true,
            coloremphstyle=\color{blue!25!black},
            empharea=r2-x6,
            setfen=%
/%
/%
2Q/%
4k17K/%
2R/%
/%
/%
]%
\hfil
\raise 10pt\hbox{\chessboard[maxfield=z7,
            printarea=r2-x6,
            tinyboard,
            labelbottom=true,
            labelleft=false,
            labelright=true,
            labelfontsize=6pt,
            showmover=false,
            border=false,
            margin=true,
            pgfstyle=straightmove,
            backmove={s3-u3},
            setfen=%
/%
/%
18Q/%
20k1K/%
20R/%
/%
/%
]%
}%
\caption{White to mate in 17 on infinite edgeless board}\label{Figure.MateIn17}
\end{figure}
With white to move, white may roll the black king to the right with {\tt 1.Re3+ Kf4 2.Qe5+ Kg4 3.Rg3+ Kh4 4.Qg5+}, leading to {\tt 16.Qs5+ Ku4 17.Ru3\#} checkmate, as shown at the right. In the general case, where the kings are separated with $n$ steps, this strategy forces check-mate in exactly $n$ steps, once the black king is rolled against the white king. Since there is no check-mate position involving just the white queen and rook, any check-mate position must bring the kings into proximity, which takes at least $n/2$ steps. So the value of the position is at least $n/2$ and at most $n$ (we believe it is exactly $n$). Thus, there are positions of unbounded finite value with these four pieces, and it follows that every value is realized during the course of play from those positions.

For positions with three pieces, there is no checkmate possibility other than two queens versus a king, since one cannot even set up a checkmate with other combinations of three pieces. But two white queens versus a black king, with white to play, is always at worst mate-in-$5$. In two moves, white should trap the black king between two files controlled by the queens, and then close the gap to allow the black king only two squares, after which there is a mating move. In figure \ref{Figure.TwoQueensMateKing}, for example, {\tt 1.Qg3+ Kf8 2.Qe2 Kf7 3.Qe9 Kf6 4.Qg4 Kf7 5.gQe6\#}.
\begin{figure}[h]
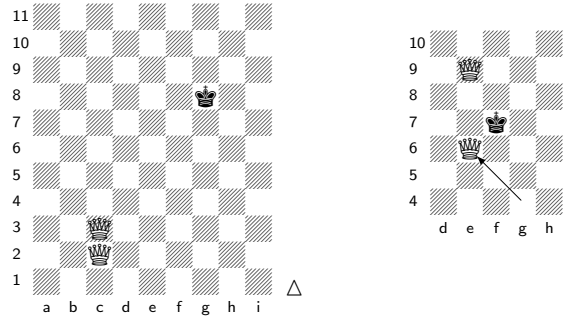

$$\chessboard[maxfield=i11,
            tinyboard,
            labelbottom=true,
            labelleft=true,
            labelfontsize=6pt, 
            labelleftwidth=1.6ex,
            mover=w,
            border=false,
            margin=true,
            setfen=%
/%
/%
/%
6k/%
/%
/%
/%
/%
2Q/%
2Q/%
/%
]%
\qquad\qquad
\raise30pt\hbox{\chessboard[maxfield=i11,
            tinyboard,
            printarea=d4-h10,
            labelbottom=true,
            labelleft=true,
            labelfontsize=6pt, 
            labelleftwidth=1.6ex,
            showmover=false,
            border=false,
            margin=true,
            pgfstyle=straightmove,
            shortenend=.7ex,
            linewidth=.05ex,
            markmove=g4-e6,
            setfen=%
/%
/%
4Q/%
/%
5k/%
4Q/%
/%
/%
/%
/%
/%
]%
}$$
\caption{White mates in $5$.}\label{Figure.TwoQueensMateKing}
\end{figure}
So with three pieces, if checkmate is possible at all, then it is at worst mate-in-$5$, and all smaller values are also realized.
\end{proof}
Consider next the position of figure \ref{Figure.ValueOmega}, which we claim has value $\omega$. Additional examples of finite positions with transfinite values are available in the answers provided by Noam D. Elkies and others to the MathOverflow question \cite{MO63423Wästlund:CheckmateInOmegaMoves}.
\begin{figure}[h]
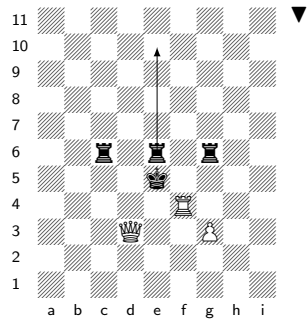

$$\chessboard[maxfield=i11,
            tinyboard,
            labelfontsize=6pt,
            labelleftwidth=2ex,
            labelbottom=true,
            labelleft=true,
            borderleft=false,
            borderright=false,
            bordertop=false,
            borderbottom=false,
            pgfstyle=straightmove,
            backmove={e6-e10},
            mover=b,
            setfen=%
/%
/%
/%
/%
/%
2r1r1r/%
4k/%
5R/%
3Q2P/%
/%
/%
]%
$$
\caption{A position with value $\omega$}\label{Figure.ValueOmega}
\end{figure}
The main line of play here calls for black to move his center rook up to arbitrary height, and then white slowly rolls the king into the rook for checkmate, similar to the pattern of play in the position of figure \ref{Figure.MateIn17}. For example, {\tt 1\ldots Re10 2.Rf5+ Ke6 3.Qd5+ Ke7 4.Rf7+ Ke8 5.Qd7+ Ke9 6.Rf9\#}. By playing the rook higher on the first move, black can force this main line of play have any desired finite length. Let us point out that any other initial move for black leads quickly to checkmate. For example, {\tt 1\ldots eRf6\ 2.Re4+ Kf5\ 3.Qd5\#}, or {\tt 1\ldots eRd6\ 2.Qf5\#}. Moving the {\tt c} rook anywhere leads to checkmate with {\tt 1\ldots cR\ 2.Rf5\#}, and similarly, moving the {\tt g} rook leads to checkmate with {\tt 1\ldots gR\ 2.Qd4\#}. So black is best off with the main line, if the goal is to delay checkmate. Finally, we argue that white has no quicker checkmate than the main line. Of course, once black moves the {\tt e} rook up, then white could skewer it with {\tt 1\ldots eRe$n$\ 2.Re4+ Kf6\ 3.Rxe$n$}, but by moving his rook off of the {\tt f}-file, he frees black's king to escape the checkmating net between the {\tt d} and {\tt f} files. In this line, black can sacrifice his rook for the {\tt g3} pawn and then simply aim to sacrifice all of his pieces, as Rook and Queen cannot alone checkmate the naked king. White might try other moves, but they would either permit the king to escape the {\tt d} and {\tt f} file checkmating net (for example, {\tt 2.Qe3+ Kd6} with the idea of {\tt Kc7}) or would simply waste time, permitting black to stop the roller with {\tt 2\ldots Rg5!}. Note that one may freely add extra black rooks to the position on the 6th rank, on the 8th rank and indeed on any even numbered rank, as indicated in figure \ref{Figure.ValueOmegaWithExtraRooks}, where the main line still works exactly as before. But for these positions, the alternatives to the main line seem even more unwise for white. Similarly, one may add extra rooks and a white king.
\begin{figure}[h]
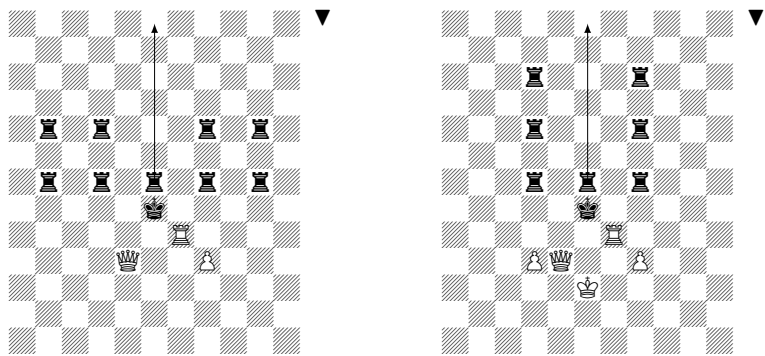

\hfil\chessboard[maxfield=k13,
            tinyboard,
            label=false,
            mover=b,
            border=false,
            margin=true,
            pgfstyle=straightmove,
            backmove={f7-f13},
            setfen=%
/%
/%
/%
/%
1r1r3r1r/%
/%
1r1r1r1r1r/%
5k/%
6R/%
4Q2P/%
/%
/%
/%
/%
]%
\hfil
\chessboard[maxfield=k13,
            tinyboard,
            label=false,
            mover=b,
            border=false,
            margin=true,
            pgfstyle=straightmove,
            backmove={f7-f13},
            setfen=%
/%
/%
3r3r/%
/%
3r3r/%
/%
3r1r1r/%
5k/%
6R/%
3PQ2P/%
5K/%
/%
/%
/%
]\hfil%
\caption{Positions with value $\omega$}\label{Figure.ValueOmegaWithExtraRooks}%
\end{figure}
The point of figure \ref{Figure.ValueOmegaWithExtraRooks} is that the main line of the position of figure \ref{Figure.ValueOmega} still works fine, since the new rooks are on the even-numbered ranks, while the white rook and queen roller checks occur on the odd-numbered ranks. But in these positions, the idea of white getting a quicker checkmate here by deviating from the main line is absurd.

Let us now introduce the omega one of chess, the theme of this article. Johan W\"astlund's question \cite{MO63423Wästlund:CheckmateInOmegaMoves} can be interpreted essentially as the question: how big is the omega one of chess $\omegaoneCh$?

\goodbreak
\begin{definition}\rm The ``{\df omega one of chess}'' refers either to the ordinal $\omegaoneCh$ or to $\omegaoneChi$, depending respectively on whether one is considering only finite positions or also positions with infinitely many pieces.
 \begin{enumerate}
  \item The ordinal $\omegaoneCh$ is the supremum of the game values of the winning finite positions for white in infinite chess.
  \item The ordinal $\omegaoneChi$ is the supremum of the game values of all the winning positions for white in infinite chess, including positions with infinitely many pieces.
  \item Similarly, $\omegaoneChthree$ and $\omegaoneChthreei$ are the analogous ordinals for infinite three-dimensional chess, as described in section \ref{Section.ThreeD}.
 \end{enumerate}
\end{definition}

There is actually an entire natural hierarchy of intermediate concepts between $\omegaoneCh$ and $\omegaoneChi$, corresponding to the various descriptive-set-theoretic complexities of the positions. For example, we may denote by $\omegaoneChc$ the {\df computable} omega one of chess, which is the supremum of the game values of the computable positions of infinite chess. Similarly, one may define $\omega_1^{\Ch,\text{hyp}}$ to be the supremum of the values of the hyperarithmetic positions and $\omega_1^{\Ch,\Delta^1_2}$ to be the supremum of the $\Delta^1_2$ definable positions, and so on.

Since there are only countably many finite positions, whose game values form an initial segment of the ordinals, it follows that $\omegaoneCh$ is a countable ordinal. The next theorem shows more, that $\omegaoneCh$ is bounded by the Church-Kleene ordinal $\omega_1^{ck}$, the supremum of the computable ordinals. Thus, the game value of any finite position in infinite chess with a game value is a computable ordinal.

\begin{theorem}\ \label{Theorem.UpperBounds}
 \begin{enumerate}
   \item $\omegaoneCh\leq\omegaoneChc\leq\omega_1^{\Ch,\text{\rm hyp}}\leq\omega_1^{ck}$. Thus, the game value of any winning finite position in infinite chess, as well as any winning computable position or winning hyperarithmetic position, is a computable ordinal. Furthermore, if a designated player has a winning strategy from a position $p$, then there is such a strategy with complexity hyperarithmetic in $p$.
   \item $\omegaoneChi\leq\omega_1$. The value of a winning position $p$ is a $p$-computable ordinal, and hence less than $\omega_1^{ck,p}$.
   \item Similarly, $\omegaoneChthree\leq\omega_1^{ck}$ and $\omegaoneChthreei\leq\omega_1$.
 \end{enumerate}
\end{theorem}

\begin{proof} Andreas Blass observed that $\omegaoneCh\leq\omega_1^{ck}$ in a MathOverflow answer \cite{MO63456Blass:CheckmateInOmegaMoves}---and Philip Welch independently made the same observation to the second author---but this also follows directly from Blass \cite{Blass1972:ComplexityOfWinningStrategies}. The point is that if there is any winning strategy for a designated player from a given position, then there is in a sense a canonical winning strategy, which is to make the closest move that minimizes the game value of the resulting position, and for a given winning finite position $p$, this strategy will have at worst hyperarithmetic complexity. To explain, let us review how the ordinal game values arise. Suppose without loss of generality that white is the designated player. Consider the game tree $T$ of all possible finite plays, that is, finite sequences of positions, each obtained by a legal move from the previous position. This is a computable tree, since we may easily check whether one position is transformed to another by a legal move. We recursively assign ordinal values to these positions when we come to see (during the course of the construction) that they are winning for white. Specifically, we identify the positions with value $\alpha$, by induction on $\alpha$. The positions with value $0$ are precisely the positions in which the game has just been completed with a win for white. If a position $p$ is black-to-play, but all legal black plays result in a position having a previously assigned ordinal value, then the value of position $p$ is decreed to be the supremum of these ordinals. If a position $p$ is white-to-play, and there is a white move to a position with value $\beta$, with $\beta$ minimal, then $p$ has value $\beta+1$. In this way, for any position $p$ with value, we may recursively associate to $p$ the value-reducing strategy $p\mapsto v(p)$, which calls for white to make a move realizing the minimum possible value (and among these choosing the resulting position with smallest possible \Godel\ code). Since playing according to the value-reducing strategy reduces value at every white move, it follows that the tree $T_p$ of all positions arising from the value-reducing strategy is well-founded, and the value of $p$ is precisely the rank of the tree $T_p$, if one should consider only the positions where it is black's turn to play.

Note that the assertion, ``position $p$ has value $\alpha$'' is $\Sigma_1$ expressible in any admissible structure containing $p$ and $\alpha$, since this is equivalent to the assertion that there is a ordinal assignment fulfilling the recursive definition of game value, which gives $p$ value $\alpha$. It follows that there can be no finite nor any computable nor indeed any hyperarithmetic position $p$ with value $\omega_1^{ck}$. To see this, suppose towards contradiction that some such position $p$ had value $\omega_1^{ck}$. Since this is a limit ordinal, it must be black's turn to play, and furthermore, every black play from $p$ to a position $q$ has an ordinal value $\beta_q$ below $\omega_1^{ck}$ and the supremum of these ordinals $\beta_q$ is $\omega_1^{ck}$. Consider the map $q\mapsto \beta_q$ in the admissible structure $L_{\omega_1^{ck}}$. Since this map is $\Sigma_1$-definable in $L_{\omega_1^{ck}}$, with a computable domain, it must be bounded in $\omega_1^{ck}$, a contradiction. So every position with a value has a value below $\omega_1^{ck}$, and so this value is therefore a computable ordinal.

If a player has a winning strategy from a position $p$, then because the ordinal game value assignment is unique and all relevant values in the game proceeding from $p$ will be bounded by the fixed value $\beta_p$ of $p$, it follows that the value-reducing strategy from $p$ is $\Delta^1_1(p)$ definable and hence hyperarithmetic in $p$.

For statements 2 and 3, the point is that from any position in infinite chess, there are only countably many legal moves for either player. Therefore, the value of any position with black to move is the supremum of a countable set of ordinals. So $\omega_1$ can never arise as such a value, and so the game values can never reach or exceed $\omega_1$. The more refined observation that the value of a position $p$ is a $p$-computable ordinal and hence at most $\omega_1^{ck,p}$ follows simply by relativizing the argument in the paragraphs above.
\end{proof}

\section{Infinite positions with transfinite game value}\label{Section.InfinitePositionsWithTransfiniteGameValue}

We should like next to present several positions of infinite chess with specific transfinite game values, thereby providing lower bounds for $\omegaoneChi$. As a warm-up, we begin with several infinite positions exhibiting what we call the basic door feature, where white aims to open a door, thereby releasing the mating material.
\begin{figure}[h]
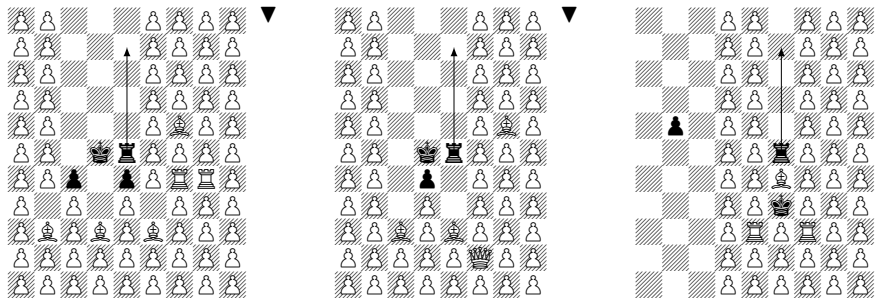

\chessboard[maxfield=i11,
           tinyboard,
            label=false,
            mover=b,
            border=false,
            margin=true,
            pgfstyle=straightmove,
            backmove={e6-e10},
            setfen=%
PP3PPPP/%
PP3PPPP/%
PP3PPPP/%
PP3PPPP/%
PP3PBPP/%
PP1krPPPP/%
PPp1pPRRP/%
P1P1P1PPP/%
PBPBPBPPP/%
PPPPPPPPP/%
PPPPPPPPP/%
]%
\hfil
\chessboard[maxfield=h11,
            tinyboard,
            label=false,
            mover=b,
            borderleft=false,
            borderright=false,
            bordertop=false,
            borderbottom=false,
            margin=true,
            pgfstyle=straightmove,
            backmove={e6-e10},
            setfen=%
PP3PPP/%
PP3PPP/%
PP3PPP/%
PP3PPP/%
PP3PBP/%
PP1krPPP/%
PP1p1PPP/%
PP1P1PPP/%
PPBPBPPP/%
PPPPPQPP/%
PPPPPPPP/%
]%
\hfil
\chessboard[maxfield=i11,
            tinyboard,
            label=false,
            mover=b,
            borderleft=false,
            borderright=false,
            bordertop=false,
            borderbottom=false,
            margin=true,
            pgfstyle=straightmove,
            backmove={f6-f10},
            setfen=%
3PP1PPP/%
3PP1PPP/%
3PP1PPP/%
3PP1PPP/%
1p1PP1PPP/%
3PPrPPP/%
3PPBPPP/%
3PPkPPP/%
3PRPRPP/%
3PPPPPP/%
3PPPPPP/%
]%
\caption{Several basic door positions with value $\omega$}
\label{Figure.SeveralPositionsValueOmega}
\end{figure}
Each of these positions, which involve infinitely many pieces (the pawn pattern is meant to continue beyond what is shown), has black to move with value $\omega$ for white. The central black rook, currently attacked by a pawn, may be moved up by black arbitrarily high, where it will be captured by a white pawn, which opens a hole in the pawn column. White may systematically advance pawns below this hole in order eventually to free up the pieces at the bottom that release the mating material. In particular, in the left position, the white pawn blocking the white rooks will eventually advance, allowing the rooks to come out and mate the black king. In the middle position, the white pawns will eventually allow the white bishop to get out of the way, which will release the white queen into the chamber to mate the black king. In the right position, the pawns will advance so as to allow one of the white rooks to mate the black king. In the right position, note the free black pawn, which prevents stalemate when the black rook is captured. Note also that in each of the positions, if black does not move his rook up on the first move, but instead captures a piece locally, then white is able to mate quickly. The value is $\omega$, since by moving his rook up to height $n$, it will take white at least $n$ moves to advance his pawns sufficiently to open the door for the mating material.

Each of the positions of figure \ref{Figure.SeveralPositionsValueOmega} is easily adapted to a position with value $\omega^2$, as in figure \ref{Figure.ValueOmegaSquared}.
\begin{figure}[h]
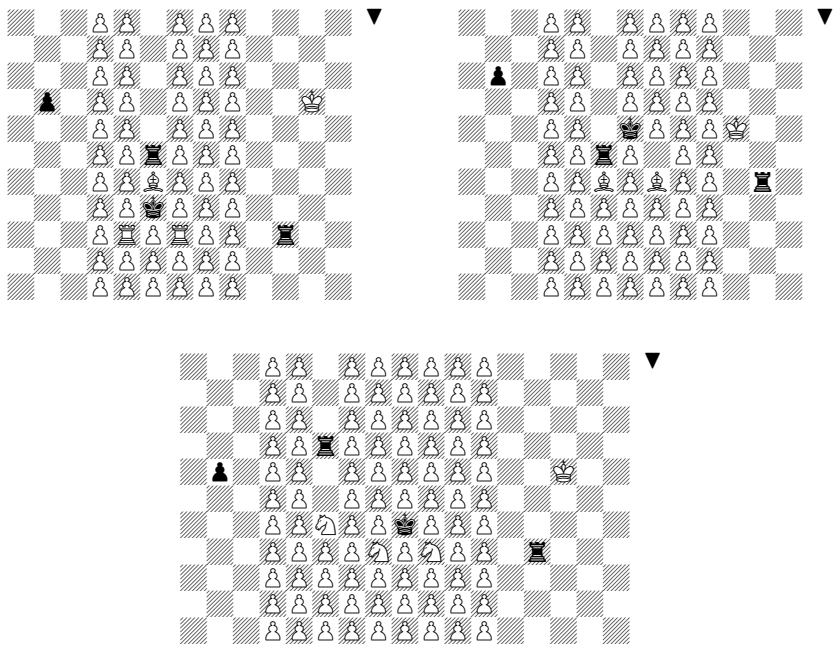

\hfil\chessboard[maxfield=m11,
            tinyboard,
            label=false,
            mover=b,
            borderleft=false,
            borderright=false,
            bordertop=false,
            borderbottom=false,
            margin=true,
            setfen=%
3PP1PPP/%
3PP1PPP/%
3PP1PPP/%
1p1PP1PPP2K/%
3PP1PPP/%
3PPrPPP/%
3PPBPPP/%
3PPkPPP/%
3PRPRPP1r/%
3PPPPPP/%
3PPPPPP/%
]%
\hfil
\chessboard[maxfield=m11,
            tinyboard,
            label=false,
            mover=b,
            borderleft=false,
            borderright=false,
            bordertop=false,
            borderbottom=false,
            margin=true,
            setfen=%
3PP1PPPP/%
3PP1PPPP/%
1p1PP1PPPP/%
3PP1PPPP/%
3PP1kPPPK/%
3PPrP1PP/%
3PPBPBPP1r/%
3PPPPPPP/%
3PPPPPPP/%
3PPPPPPP/%
3PPPPPPP/%
]\hfil%

\hfil\chessboard[maxfield=q11,
            tinyboard,
            label=false,
            mover=b,
            borderleft=false,
            borderright=false,
            bordertop=false,
            borderbottom=false,
            margin=true,
            setfen=%
3PP1PPPPPP/%
3PP1PPPPPP/%
3PP1PPPPPP/%
3PPrPPPPPP/%
1p1PP1PPPPPP2K/%
3PP1PPPPPP/%
3PPNPPkPPP/%
3PPPPNPNPP1r/%
3PPPPPPPPP/%
3PPPPPPPPP/%
3PPPPPPPPP/%
]\hfil%
\caption{Several positions with value $\omega^2$}%
\label{Figure.ValueOmegaSquared}
\end{figure}
The trick is simply to add a free black rook and the white king in a separate region, that does not interact with the door position. The idea then, is that every time white advances a pawn in the pawn column, black will proceed to harass the white king by systematic checking with the free black rook. White cannot allow himself to be in perpetual check, and so the only alternative is for white, when moving out of these harassing checks, to advance towards the free black rook. If the free black rook initially checks the white king at a distance of $k$, then white will move out of check but towards the black rook, now at distance $k-1$. Black can continue this harassing behavior for $k$ steps, but eventually the white king will be positioned to capture the free black rook, and so black now responds by moving an arbitrary distance away (the new value of $k$), giving white the opportunity to advance one white pawn in the pawn column, before the subsequent round of harassing checks begins.

To summarize, the main line of play is as follows:  black moves his inner rook up to height $n$. White aims to advance the pawn column $n$ times. But for each pawn movement, black has a chance to move his free rook out an arbitrary distance $k$, and making essentially $k$ many harassing checks on the white king. This pattern of play amounts to value $\omega^2$, since black gets to announce the value $n$, which is the number of times an additional announcement $k$ will be made, where $k$ is the number of moves that will occur before the next announcement. This pattern of play shows that the position has value at least $\omega^2$; the fact that any deviation from this pattern of play leads to a quicker checkmate for white means that the value is precisely $\omega^2$.

Consider the following somewhat sparser position, in figure \ref{Figure.ReleasingTheHordes}, with essentially the same features as the positions in figure \ref{Figure.ValueOmegaSquared}. The pawn column continues infinitely up and down, dividing the plane into two halves.
\begin{figure}[h]
\hfil\chessboard[maxfield=o13,
            tinyboard,
            label=false,
            mover=b,
            borderleft=false,
            borderright=false,
            bordertop=false,
            borderbottom=false,
            margin=true,
            pgfstyle=straightmove,
            backmove={f9-f13},
            emphstyle=\color{red!50!black},
            emphfield=f6,
            setfen=%
4P1P/%
4P1P/%
4P1P4r/%
4P1P/%
4PrP/%
Q3PBP3K/%
QQ2PPP/%
QQQ2P/%
QQ3P1k/%
Q4P/%
5P/%
5P/%
5P/%
]\hfil%
\caption{Releasing the hordes, black to move, value $\omega^2$}%
\label{Figure.ReleasingTheHordes}
\end{figure}
White aims to open the portcullis---the emphasized pawn below the central bishop---thereby releasing the queen horde into the right half-plane, which can proceed to deliver checkmate. The pattern of play here is that black will move his center rook upward an arbitrary height $n$, after which white aims to advance the pawn column up in order to free the central keystone bishop, which keeps the door closed until it is moved. Each white pawn advance gives black an opportunity to make another round of harassing checks on the white king in the right half plane. Note that black does not have mating material in the right half plane, but if the door should ever open, then the queens will easily overwhelm the black rook and make checkmate on the black king in the right chamber. The overall value is $\omega^2$, just as in the position in figure \ref{Figure.ValueOmegaSquared}. Note that for black to sacrifice his rook in the right chamber by capturing a pawn in the wall will simply open a hole in the wall earlier than necessary. Similarly, sacrificing his rook to the white king will simply omit his chance to initiate the harassing checks, thereby reducing the value to $\omega$. It is not possible to use the black rook and black king to make perpetual checks on the white king in the right chamber.

Let us now push this a bit harder, by having not just one black rook column, but an iterated series of black rook columns. The idea is to use what we call a locked-door-and-key arrangement, for which there is a black rook tower, but it becomes activated only after certain white attacking pieces are released. To illustrate the basic idea, consider the positions in figure \ref{Figure.BasicLockAndKey}.
\begin{figure}[h]
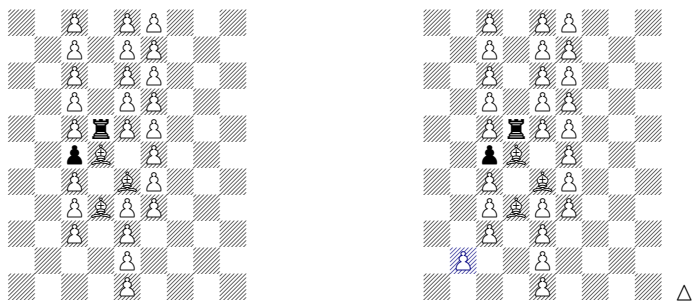

\hfil\chessboard[maxfield=i11,
            tinyboard,
            label=false,
            showmover=false,
            borderleft=false,
            borderright=false,
            bordertop=false,
            borderbottom=false,
            margin=true,
            setfen=%
2P1PP/%
2P1PP/%
2P1PP/%
2P1PP/%
2PrPP/%
2pB1P/%
2P1BP/%
2PBPP/%
2P1P/%
4P/%
4P/%
]%
\hfil
\chessboard[maxfield=i11,
            tinyboard,
            label=false,
            mover=w,
            borderleft=false,
            borderright=false,
            bordertop=false,
            borderbottom=false,
            margin=true,
            emphstyle=\color{blue!50!black},
            emphfield=b2,
            setfen=%
2P1PP/%
2P1PP/%
2P1PP/%
2P1PP/%
2PrPP/%
2pB1P/%
2P1BP/%
2PBPP/%
2P1P/%
1P2P/%
4P/%
]\hfil%
\caption{Locked door at left; locked door with key at right}%
\label{Figure.BasicLockAndKey}%
\end{figure}
The position at left shows the locked door position, which is stable in the sense that the black rook is not forced to move (assuming black can move elsewhere). The position at right shows the same position, but with the additional (emphasized) free white pawn, the ``key'' pawn, which white may advance to attack the black pawn, the ``lock,'' which will activate the tower, thereby ultimately unlocking the door. When the key white pawn attacks the black pawn, black will not capture, since this will free the bishop and open the door right away. Rather, black will let white capture the black pawn---inserting the key into the lock---which then attacks the black rook, which will then advance up to an arbitrary height, forcing white to advance the pawn column before the door opens.

By chaining together an iterated series of these locked doors, we produce positions with value $\omega^2\cdot k$ for any finite $k$. Consider the position in figure \ref{Figure.ValueOmegaSquareTimes4}, which chains together four locked doors and has value $\omega^2\cdot 4$.
\begin{figure}[h]
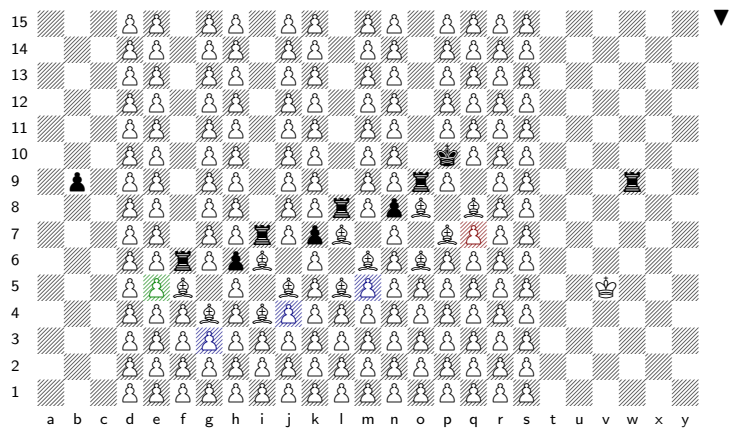

\hfil
\chessboard[maxfield=y15,
            tinyboard,
            labelbottom=true,
            labelleft=true,
            labelfontsize=6pt,
            labelleftwidth=2ex,
            mover=b,
            borderleft=false,
            borderright=false,
            bordertop=false,
            borderbottom=false,
            margin=true,
            emphstyle=\color{green!50!black},
            emphfield=e5,
            emphstyle=\color{blue!50!black},
            emphfields={g3,j4,m5},
            emphstyle=\color{red!50!black},
            emphfield=q7,
            setfen=%
3PP1PP1PP1PP1PPPP/%
3PP1PP1PP1PP1PPPP/%
3PP1PP1PP1PP1PPPP/%
3PP1PP1PP1PP1PPPP/%
3PP1PP1PP1PP1PPPP/%
3PP1PP1PP1PP1kPPP/%
1p1PP1PP1PP1PPrP1PP3r/%
3PP1PP1PPrPpB1BPP/%
3PP1PPrPpB1P1BPPP/%
3PPrPpB1P1BPBPPPP/%
3PPB1P1BPBPPPPPPP2K/%
3PPPBPBPPPPPPPPPP/%
3PPPPPPPPPPPPPPPP/%
3PPPPPPPPPPPPPPPP/%
3PPPPPPPPPPPPPPPP/%
]\hfil%
\caption{A position with value $\omega^2\cdot 4$}%
\label{Figure.ValueOmegaSquareTimes4}
\end{figure}
The main line of play here is that black moves up his first rook at {\tt f6}, currently attacked by the key pawn at {\tt e5}, to an arbitrary height $n$, his first large announcement. White takes the rook with a pawn, and aims to advance the corresponding pawn column successively so as to free up the {\tt f} and {\tt g} bishops below, which will release the next key white pawn at {\tt g3}, which will advance and attack the black lock pawn at {\tt h6}, which guards the second black rook tower. This will activate the second tower, and black moves the rook at {\tt i7} up to any desired height, the second large announcement, and so on for each of the four towers. Each key white pawn ultimately activates the corresponding black rook tower, with black replying by moving the rook up to arbitrary height, making four large announcements in all. When the white pawns of the final tower advance to release the bishops below, the final advancing key pawn at {\tt q7} advances to delivers checkmate to the black king. Note that with every single white pawn advance in any of the columns, black may make an arbitrarily long campaign of harassing checks on the white king in the free region at the right. In other words, the large announcements correspond to the number of white tower pawn advances that will need to be made by white, but each tower pawn advance enables black to make a small announcement, by moving his free black rook out to arbitrary distance and initiating a campaign of harassing checks. Thus, whenever possible, black checks the white king, unless his rook would be under attack by the king, in which case he moves it an arbitrary distance out (a small announcement), and for this brief reprieve from harassment, white is able to make one additional white pawn advancement before the next extremely long round of harassment begins. Any deviation from this pattern of play by black leads to a quicker checkmate, and white is unable to force on his or her own a quicker mate. The value is $\omega^2\cdot 4$ because black gets to make four large announcements, one for each black rook tower, which each provide the number of additional small announcements that will be made (one for each white pawn advance in that tower), for each of which he will be able to make at least that many moves (via harassing checks) before the next announcement. It is clear that one may use the same idea to construct positions with value $\omega^2\cdot k$ for any fixed $k$, simply by the continuing the pattern to have $k$ black rook towers.

We shall now push this idea a bit harder in order to produce a position with value $\omega^3$. The idea here is that we want to allow an arbitrary (but definitely finite) number of black rook towers to become activated.
\begin{figure}[h]
\hfil
\chessboard[maxfield=z29,
            boardfontsize=9pt,
            labelbottom=true,
            labelleft=true,
            labelfontsize=6pt,
            labelleftwidth=2ex,
            showmover=false,
            borderleft=false,
            borderright=false,
            bordertop=false,
            borderbottom=false,
            coloremphstyle=\color{green!25!black},
            emphfield=m6,
            coloremphstyle=\color{blue!50!black},
            emphfields={n12,q15,t18,w21,z24},
            coloremphstyle=\color{red!50!black},
            emphfield=k9,
            pgfstyle=straightmove,linewidth=.05ex,
            backmove=m6-p9,
            margin=false,
            setfen=%
6PPPPP1PP1PP1PP1PP1PB/%
6PPPPP1PP1PP1PP1PPrP1/%
6PPPPP1PP1PP1PP1P1BpB/%
6PPPPP1PP1PP1PP1PB1P1/%
6PPPPP1PP1PP1PPrP1BPB/%
6PPPPP1PP1PP1P1BpB1PP/%
6PPPPP1PP1PP1PB1P1BPB/%
6PPPPP1PP1PPrP1BPB1Pp/%
6PPPPP1PP1P1BpB1PPBpP/%
6PPPPP1PP1PB1P1BPBpP1/%
2K3PPPPP1PPrP1BPB1PpP2p/%
6PPPPP1P1BpB1PPBpP2pP/%
6PPPPP1PB1P1BPBpP2pP/%
6PPPPPrP1BPB1PpP2pP/%
6PPPP1BpB1PPBpP2pP/%
6PPPPB1P1BPBpP2pP/%
3r2PPPP1BPB1PpP2pP/%
6PPPkB1PPBpP2pP/%
6PP1P1BPBpP2pP/%
6PPBPB1PpP2pP/%
6PPPPPBpP2pP/%
6PPPPBpP2pP/%
6PPPP1P2pP/%
6PPPPBPbpP/%
6PPPPPP1p/%
6PPPPPP1P/%
6PPPPPP/%
6PPPPPP/%
6PPPPPP/%
]%
\chessboard[maxfield=m29,
            boardfontsize=9pt,
            label=false,
            showmover=true,
            mover=b,
            borderleft=false,
            borderright=false,
            bordertop=false,
            borderbottom=false,
            margin=false,
            coloremphstyle=\color{blue!50!black},
            emphfield=c27,
            setfen=%
1P1BPBpP2pP/%
BPB1PpP2pP/%
1PPBpP2pP/%
BPBpP2pP/%
1PpP2pP/%
BpP2pP/%
pP2pP/%
P2pP/%
2pP/%
1pP/%
pP/%
P/%
]%
\caption{A position with value $\omega^3$}%
\label{Figure.ValueOmegaCubed}
\end{figure}
In the first move, black in effect will choose to have access to $k$ black rook towers, and the resulting position will have value $\omega^2\cdot k$ as above. If black is able to choose arbitrarily large finite $k$, the overall value will therefore be $\omega^3$. Consider the position in figure \ref{Figure.ValueOmegaCubed}. The idea here is that the black bishop at {\tt m6}, currently under attack, goes up the diagonal, is captured, and this releases a white pawn, the initial key pawn, which will activate the black rook towers in sequence. Within each rook tower, the white key pawns (emphasized) move to attack the black lock pawn, guarding the towers; black moves his rook up an arbitrary amount, and white aims to advance the white pawn column so as to released the bishops below the tower, which will release the next white key pawn activating the next tower, the adjacent lower tower to the left (so the towers are activated from right to left now). As in the position of figure \ref{Figure.ValueOmegaSquareTimes4}, with each white pawn advance in the tower, black is able to initiate a campaign of harassing checks on the white king in the free region, now at the left. For example, {\tt 1\ldots mBu14 2.xu14} will lead, after a campaign of harrassing checks, to {\tt 3.u15}, which will eventually release the {\tt t} and {\tt u} bishops, thereby allowing the key pawn at {\tt t18}, which will attack the black lock pawn at {\tt s21}. The key white pawn will capture the black lock pawn, and black will move the rook at {\tt r22} up to any desired height, and this tower is active. White aims to advance the white pawns in the {\tt q} file, in order ultimately to release the {\tt q} and {\tt r} bishops below, which will release the key pawn at {\tt q15}, advancing to attack the lock pawn at {\tt p18} and activating the next tower. As before, with each white pawn advance, black embarks on an arbitrarily long campaign of harassing checks on the white king in the free region at the left, which must be completed before white can take another step in the main line. Eventually, when the last rook tower is activated, on the {\tt l} file, and the white pawns advance in the {\tt k} file, the {\tt k} and {\tt l} bishops are released, allowing at long last the key pawn at {\tt k9} to advance and force the black king into checkmate via {\tt $\<n>$.k11+ Kk13\ $\<n+1>$.j12\#}.

We claim that the main line of play we have just described shows that the position is winning for white, with value $\omega^3$. Black can move his bishop out past the $k^{\rm th}$ rook tower, for any finite value of $k$ he desires, and thereby play from a position in which there are $k$ rook towers, a position with value at least $\omega^2\cdot k$, giving value $\omega^3$ to the original position.

It remains for us to argue that the main line of play is essentially the only viable line of play, meaning that white cannot make a quicker checkmate than we have described, and black cannot realize higher values. So let us consider several issues.

First, note that if we were to remove the black bishop at {\tt m6}, then the position would become drawn, since white would have no way to activate any of the rook towers, and black would simply move his rook perpetually in the free part of the board. Similarly, in the given position, if the {\tt k} and {\tt l} files remain locked up, then there is clearly no checkmate possible, so white must aspire to activate the {\tt l} rook tower. Moving inductively to the right, if the {\tt n} and {\tt o} bishops are never released, then there is no hope of activating the {\tt l} tower, and so white must first activate the {\tt o} rook tower. And so on. So white must aspire to activate one of the towers, preferring it to be as close to the king as possible. For this reason, black has no interest in self-activating any of the towers, by moving a tower rook up even when it is not yet attacked, since this simply starts the tower process prematurely.

For the first move, note that if black does not move the bishop at {\tt m6}, then white could capture it with {\tt 1\ldots xm6 2.Bl5}, freeing the pawn at {\tt k5}, which could advance and attack {\tt l8}, which would quickly release the {\tt k} and {\tt l} bishops, enabling the key pawn at {\tt k9} to advance with checkmate. So on his first move black will want to move the black bishop. Note that capturing at {\tt l5} or {\tt l7} has a similar effect, and so instead black will move out on the main diagonal, following the main line. White is clearly best off capturing it, before it moves further out, which would increase value.

When a given black tower is activated, note that black has no interest in any movement except moving the black rook upward, since capturing down, left or right simply releases the bishops under the tower more quickly. So when a key white pawn is inserted into the lock, then black will indeed move his rook upward. Once a tower is activated in this way, then black still has no interest in self-activating any of the {\it earlier} towers, that is, to the left. Rather, he prefers to force white to undergo the long campaigns of harassing checks, in order to slowly advance his tower pawns before releasing the next key pawn.

But let us point out that there is an interesting plan that black can adopt, attempting to prolong the length of play. Namely, once a tower is activated, then black is also free to move rooks in the irrelevant towers further out, beyond the current active tower. Although any unprovoked move by a black rook in any of those towers can be immediately answered by white capturing it, nevertheless the interesting thing here is that white cannot actually afford always to answer such moves by black. The reason is that if white were to always answer such moves, then as there are infinitely many such towers available, black will simply continue to make such moves, with white answering, leading to a draw by infinite play. Consequently, white must therefore sometimes ignore such black moves, in order to make progress on the main line. Knowing this, perhaps what black hopes to do is to amass an army of black rooks just outside the active tower, which will then invade all together and cause chaos in the mating line. We argue that white needn't let things come near to that. What we claim is that white can plan to answer most of the irrelevant black rook movement, ignoring it only every tenth time, for example, and never ignoring rooks on adjacent towers. In the case that an irrelevant rook movement is ignored by white, then it is easy to see that white can capture it upon any subsequent movement, except if its second movement should bring it to one of the five emphasized squares indicated in figure \ref{Figure.IrrelevantTowers}.
\begin{figure}[h]
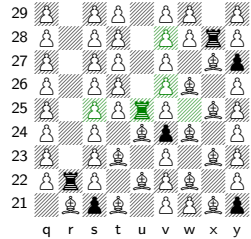

\hfil\chessboard[maxfield=z29,
            printarea=q21-y29,
            boardfontsize=9pt,
            labelbottom=true,
            labelleft=true,
            labelfontsize=6pt,
            labelleftwidth=2ex,
            showmover=false,
            borderleft=false,
            borderright=false,
            bordertop=false,
            borderbottom=false,
            coloremphstyle=\color{green!50!black},
            emphfields={s25,w25,v26,u25,v28},
            margin=false,
            setfen=%
6PPPPP1PP1PP1PP1PP1PB/%
6PPPPP1PP1PP1PP1PPrP1/%
6PPPPP1PP1PP1PP1P1BpB/%
6PPPPP1PP1PP1PP1PB1P1/%
6PPPPP1PP1PP1PPrP1BPB/%
6PPPPP1PP1PP1P1BpB1PP/%
6PPPPP1PP1PP1PB1P1BPB/%
6PPPPP1PP1PPrP1BPB1Pp/%
6PPPPP1PP1P1BpB1PPBpP/%
6PPPPP1PP1PB1P1BPBpP1/%
2K3PPPPP1PPrP1BPB1PpP2p/%
6PPPPP1P1BpB1PPBpP2pP/%
6PPPPP1PB1P1BPBpP2pP/%
6PPPPPrP1BPB1PpP2pP/%
6PPPP1BpB1PPBpP2pP/%
6PPPPB1P1BPBpP2pP/%
3r2PPPP1BPB1PpP2pP/%
6PPPkB1PPBpP2pP/%
6PP1P1BPBpP2pP/%
6PPBPB1PpP2pP/%
6PPPPPBpP2pP/%
6PPPPBpP2pP/%
6PPPP1P2pP/%
6PPPPBPbpP/%
6PPPPPP1p/%
6PPPPPP1P/%
6PPPPPP/%
6PPPPPP/%
6PPPPPP/%
]\hfil%
\caption{Irrelevant rook movement in distant towers}\label{Figure.IrrelevantTowers}
\end{figure}
But notice that in this case, white can ensure that the black rook can be captured upon any subsequent movement. For example, if on the ignored move black has captured either the pawn at {\tt t25} or {\tt v25}, and then subsequently moved to one of the emphasized squares, then white can move his bishop from {\tt u24} to the square where the captured pawn had been, thereby trapping the rook in the sense that any further movement will result in immediate capture. In this way, after having ignored an irrelevant rook movement, white can ensure that this rook would be captured upon its third movement, essentially without leaving that tower. It follows that black cannot hope in this way to bring his rooks from the distant towers into any useful action, and furthermore, we claim that black actually has no interest in pursuing any of this irrelevant rook activity, because it was ultimately answerable by white by causing at most a bounded number of extra moves, a bounded increase in value, whereas black had to give up the opportunity to undertake a campaign of harassing checks in order to move the rook in the first place, with value of $\omega$. Basically, whatever delay black hoped to cause by means of the irrelevant rook movements, he could have caused a greater delay by moving his rook out further in the free region of the board.

Finally, let us consider one of the consequences of the main line of play, the fact that activating a rook tower causes the release of a number of white bishops, which can find their way into the spare area of the board at the lower right. Indeed, by using his key pawn to activate not only one tower but also the (irrelevant) towers beyond that tower---and this is indeed possible---white can cause the release of any desired number of white bishops into the spare area in the lower right. We observe simply that this doesn't matter, since they are all white square bishops, and there is nothing for them to do in that area. White square white bishops in the spare area at the lower right simply have nothing to attack, except the black pawn at {\tt n5}, which is of no consequence (it is there merely to help control the very first move). In particular, having extra white square white bishops in the lower right will not enable white to activate any tower that is not yet activated, and so this artifact of the position is nothing to worry about.

So we have presented a position of infinite chess with game value $\omega^3$. Let us briefly describe how to modify the position of figure \ref{Figure.ValueOmegaCubed} to produce positions of value $\omega^3\cdot 2$ and more. The idea is to add another sequence of black rook towers in the upper left part of the board. With value $\omega^3$, that part of the position will end not with checkmate, but rather by releasing in a certain precise way the key bishop that activates the current position. That is, rather than starting with the bishop at {\tt m6} as in figure \ref{Figure.ValueOmegaCubed}, white has to undergo an $\omega^3$ game in order to force black into the position of needing to move his bishop. One can imagine a black pawn lock guarding the black bishop at {\tt m6}, while another $\omega^3$ position at the left will have the effect of releasing a key white pawn to open that lock, causing at that point the black bishop to move. Such a position would have overall game value $\omega^3\cdot 2$.

We should like also to observe that one may combine a version of the position of figure \ref{Figure.ValueOmegaCubed} in the upper-right half plane with a reversed-color version of the position in lower-left half plane, in order to produce a position having value $\omega^3$ for the {\it second} player, whoever that is decreed to be. One should modify the positions so that instead of making checkmate directly after the last rook tower, instead what happens is a door opens releasing the mating material into the harassment region, having a separate harassment region for each king (these may be confined to columns of finite width).

The positions in this section show that $\omega^3<\omegaoneChi$, and we have indicated how to show $\omega^3\cdot 2<\omegaoneChi$, but we conjecture that $\omegaoneChi$ is far larger than this. We believe that chess is sufficiently complex so as to realize the full range of game values.

\begin{conjecture}\label{Conjecture.OmegaOneCH=OmegaOne}
 The omega one of chess, both for finite positions and for arbitrary positions, is as large as it can be:
 $$\omegaoneCh=\omegaoneChc=\omega_1^{\Ch,\text{\rm hyp}}=\omega_1^{ck}\qquad\qquad\omegaoneChi=\omega_1.$$
\end{conjecture}

The first equation is equivalent merely to $\omegaoneCh=\omega_1^{ck}$, since the other values are trapped between. This part of the conjecture, if true, implies that there should be finite positions in infinite chess with extremely large ordinal game values. Evidence against this might be the fact that the best lower bounds currently known are comparatively small. For example, amongst finite positions the best we have observed in this article is that $\omega<\omegaoneCh$, by means of the position of figure \ref{Figure.ValueOmega}, and the highest values currently claimed in the answers to W\"astlund's question \cite{MO63423Wästlund:CheckmateInOmegaMoves} are $\omega\cdot n$. So although we believe that finite high-value positions exist, we don't yet have specific examples. We suggest that one way to make a finite position with value $\omega^2$ might be for white to have an inefficient mating net, perhaps with four bishops, which black can initially push arbitrarily far out; the idea would be that the inefficiency in carrying out the mate will give black free moves in order to carry on a checking campaign in the manner of the black rook on the side in the positions of figure \ref{Figure.ValueOmegaSquared}. Although we've considered several promising candidates, none so far has fully succeeded. Meanwhile, the infinite positions of section \ref{Section.InfinitePositionsWithTransfiniteGameValue} are all computable, and so at least we know $\omega^3<\omegaoneChc$.

Some evidence in favor of the conjecture is the fact that theorem \ref{Theorem.3DAnyValueAtttained} and corollary \ref{Corollary.OmegaOneCh3c} establish an analogue of it in the case of infinite three-dimensional chess, by embedding any well-founded infinitely branching tree in an infinite position of infinite three-dimensional chess. One could conceivably aspire to prove that $\omegaoneChi=\omega_1$ by finding similar embeddings of such trees in two-dimensions, but our attempts to carry this idea out have faced serious chess geometry difficulties; it seems that infinite chess in two dimensions is cramped.

\section{Computable strategies}\label{Section.ComputableStrategies}

In this section, we prove that for computable infinite positions in infinite chess, the question of whether the players are required to play according to a deterministic computable procedure can affect our judgement of whether a given computable position is a win or a draw.

\begin{theorem}\label{Theorem.ComputableTree}
 There is a position in infinite chess such that:
  \begin{enumerate}
   \item The position is computable.
   \item In the category of computable play, the position is a win for white. White has a computable strategy defeating every
     computable strategy for black.
   \item However, in the category of arbitrary play, the position is a draw. Neither player has a winning strategy, and either player can force a draw.
  \end{enumerate}
\end{theorem}

\begin{proof}
Let $T\subset 2^{\ltomega}$ be any computable infinite tree having no computable infinite branch. Such trees are known to exist by elementary classical results of computability theory. Since the tree is infinite, it is an infinite finitely-branching tree, and hence by K\"onig's lemma it has an infinite branch, even if it has no infinite computable branch.

Consider the infinitary chess position resulting from making a copy of the tree $T$ in the following manner:
\begin{figure}[h]
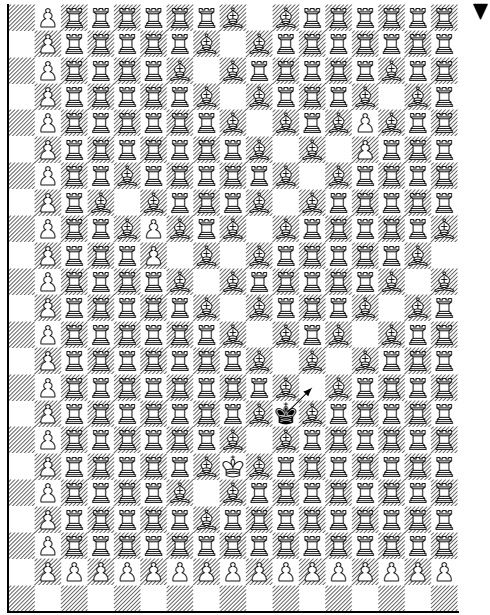

\hfil\chessboard[maxfield=q23,
            tinyboard,
            label=false,
            mover=b,
            borderleft=true,
            borderright=false,
            bordertop=false,
            borderbottom=true,
            pgfstyle=straightmove,
            backmove=k8-l9,
            setfen=%
1PRRRRRRB1BRRRRRR/%
1PRRRRRB1BRRRRRRR/%
1PRRRRB1BRRRRRBRR/%
1PRRRRRB1BRRRB1BR/%
1PRRRRRRB1BRBPBRR/%
1PRRRRRRRB1B1PRRR/%
1PRRBRRRRRB1BRRRR/%
1PRB1BRRRB1BRRRRR/%
1PRRBPBRB1BRRRRRB/%
1PRRRP1B1BRRRRRB1/%
1PRRRRB1BRRRRRB1B/%
1PRRRRRB1BRRRB1BR/%
1PRRRRRRB1BRB1BRR/%
1PRRRRRRRB1B1BRRR/%
1PRRRRRRRRB1BRRRR/%
1PRRRRRRRBkBRRRRR/%
1PRRRRRRB1BRRRRRR/%
1PRRRRRBKBRRRRRRR/%
1PRRRRB1BRRRRRRRR/%
1PRRRRRBRRRRRRRRR/%
1PRRRRRRRRRRRRRRR/%
1PPPPPPPPPPPPPPPP/%
/%
]\hfil
\caption{White forces black to climb the tree; black hopes to avoid the dead-end traps.}\label{Figure.TreeChannels}
\end{figure}
The position consists of a series of channels, branching points and dead-ends, arranged in the form of a tree, whose tree structure coincides with $T$. The channels consist of empty white squares on a diagonal, lined with bishops, which are in turn protected by rooks, which also protect each other. None of the bishops and rooks have any legal moves. Only the lower portion of the position is shown, of course, but all the necessary details of branching and turning corners and leaves (dead-ends) are indicated, and it is clear that we may arrange such a position so that the channels have the structure of any given binary tree, such as the tree $T$ that we fixed at the beginning of this proof. The position is computable, because the tree $T$ is computable, and the position may be built in a computable modular fashion from purely local information about which binary sequences are in $T$.

Note that the kings can move freely through the channel, and initially, the only moves for either player are with the king. The position has black to play, and there is only one move, which is for black to move his king up the channel. As long as black is moving through the channel, white can move only his king.

If the tree $T$ has an infinite branch, then it is clear that black has a drawing strategy: simply climb the tree along an infinite branch. This will avoid all the dead-ends; no captures need ever be made, and play will continue indefinitely. Hence, it leads to a draw by indefinite play. And since black clearly cannot check-mate white with the resources available, the position is drawn for both players.

The interesting thing to notice about this drawing strategy, however, is that it makes use of a non-computable resource, namely, the infinite branch through the tree, which by the choice of $T$ we ensured is necessarily non-computable. So let us consider the situation where black plays according to a deterministic computable procedure. White will play simply to follow the black king, eliminating the possibility of black retreating in the tree. Thus, essentially by the principle of zugzwang, white can force black to climb the tree. If black is playing according to a computable procedure, then since the tree has no computable infinite branch, black will inevitably find himself in one of the dead-end zones. Note with detail what happens when black enters such a zone: the black king climbs into the zone, followed by the white king close behind; the black king captures a pawn, followed by the white king again; the black king moves finally to the very last cell of the dead-end zone, and white pushes the pawn for check-mate.
\end{proof}

The previous position worked by the principle of zugzwang, using the white king to push the black king through the tree, relying on the fact that black has no other move and must therefore climb the tree. When combining this tree idea with other positions, however, it will be convenient
to have a version of it in which white can more forcefully push black through the tree. So consider the alternative positions of figure \ref{Figure.ForcingBlackToClimb}. The one on the left is obtained by filling the channel with pawns, and swapping the bishops for pawns.
\begin{figure}[h]
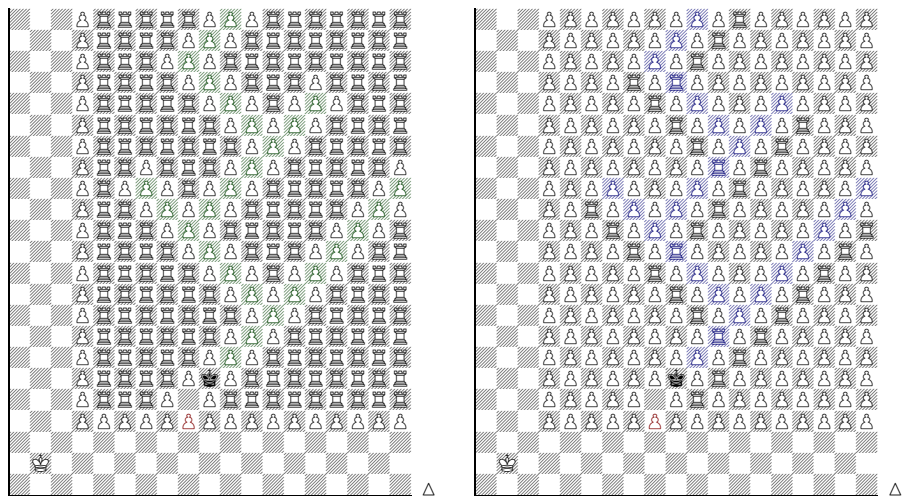

\hfil\chessboard[maxfield=s23,
            boardfontsize=8pt,
            label=false,
            mover=w,
            borderleft=true,
            borderright=false,
            bordertop=false,
            borderbottom=true,
            coloremphstyle=\color{green!25!black},
            emphfields={k7,l8,m9, n10, o11, p12, q13, r14, s15, l10, k11, j12, i13, h14, g15, j14, k15, l16, m17, n18, o19, l18, k19, j20, i21, j22, k23},
            coloremphstyle=\color{red!50!black},
            emphfield=i4,
            setfen=%
3PRRRRRPPPRRRRRRR/%
3PRRRRPPPRRRRRRRR/%
3PRRRPPPRRRRRRRRR/%
3PRRRRPPPRRRPRRRR/%
3PRRRRRPPPRPPPRRR/%
3PRRRRRRPPPPPRRRR/%
3PRRRRRRRPPPRRRRR/%
3PRRPRRRPPPRRRRRP/%
3PRPPPRPPPRRRRRPP/%
3PRRPPPPPRRRRRPPP/%
3PRRRPPPRRRRRPPPR/%
3PRRRRPPPRRRPPPRR/%
3PRRRRRPPPRPPPRRR/%
3PRRRRRRPPPPPRRRR/%
3PRRRRRRRPPPRRRRR/%
3PRRRRRRPPPRRRRRR/%
3PRRRRRPPPRRRRRRR/%
3PRRRRPkPRRRRRRRR/%
3PRRRP1PRRRRRRRRR/%
3PPPPPPPPPPPPPPPP/%
/%
1K/%
/%
]%
\hfil
\chessboard[maxfield=s23,
            boardfontsize=8pt,
            label=false,
            mover=w,
            borderleft=true,
            borderright=false,
            bordertop=false,
            borderbottom=true,
            coloremphstyle=\color{blue!45!black},
            emphfields={k7,l8,m9, n10, o11, p12, q13, r14, s15, l10, k11, j12, i13, h14, g15, j14, k15, l16, m17, n18, o19, l18, k19, j20, i21, j22, k23},
            coloremphstyle=\color{red!50!black},
            emphfield=i4,
            setfen=%
3PPPPPPPPPRPPPPPP/%
3PPPPPPPPRPPPPPPP/%
3PPPPPPPRPPPPPPPP/%
3PPPPRPRPPPPPPPPP/%
3PPPPPRPPPPPPPPPP/%
3PPPPPPRPPPPPRPPP/%
3PPPPPPPRPPPRPPPP/%
3PPPPPPPPRPRPPPPP/%
3PPPPPPPPPRPPPPPP/%
3PPRPPPPPRPPPPPPP/%
3PPPRPPPRPPPPPPPR/%
3PPPPRPRPPPPPPPRP/%
3PPPPPRPPPPPPPRPP/%
3PPPPPPRPPPPPRPPP/%
3PPPPPPPRPPPRPPPP/%
3PPPPPPPPRPRPPPPP/%
3PPPPPPPPPRPPPPPP/%
3PPPPPPkPRPPPPPPP/%
3PPPPP1PRPPPPPPPP/%
3PPPPPPPPPPPPPPPP/%
/%
1K/%
/%
]\hfil%
\caption{White forces black to climb the tree with successive pawn checks, two variations}\label{Figure.ForcingBlackToClimb}
\end{figure}
At the right, a somewhat cleaner version of the position with the same essential features is made using only rooks and pawns. The king can be forced by successive pawn checks to climb the tree of highlighted squares. These positions also prove the claims made in theorem \ref{Theorem.ComputableTree}, as well as the additional claims of theorem \ref{Theorem.RookPawnTree}. The difference here is that white pushes the black king through the tree by pushing a pawn at each step, giving check. The only legal move for black is to climb the tree, in which case the next pawn is available for white to push with check.

\break
\begin{theorem}\label{Theorem.RookPawnTree}
There is a position in infinite chess such that:
 \begin{enumerate}
   \item The position is computable.
   \item In the category of computable play, the position is a win for white. White has a computable strategy defeating every
     computable strategy for black.
   \item Furthermore, white has such a computable strategy in which every move is check, defeating every computable strategy for black.
   \item White has another such computable strategy, defeating every computable strategy of black, in which white simply moves his king back and forth until the last four moves forcing checkmate.
   \item Meanwhile, in the category of arbitrary play, the position is drawn. Either player may force a draw. White has a computable drawing strategy, which does not depend on the moves of black. Black has a drawing strategy of `low' Turing complexity, whose moves do not depend on the moves of white.
\end{enumerate}
\end{theorem}

\begin{proof}
Consider either of the positions of figure \ref{Figure.ForcingBlackToClimb}, where the structure of the corresponding tree $T$ is an infinite computable tree with no computable infinite branches, as in theorem \ref{Theorem.ComputableTree}. These are computable positions, and by simply climbing the tree, black can avoid check-mate. Since there is no computable infinite branch through the tree, however, every computable strategy for black must eventually find itself in one of the dead-end nodes, where black will be check-mated. White may force black upward in the tree by means of successive pawn checks, which culminate in check-mate in a dead-end. Alternatively, white may simply move his king back and forth, waiting either for black to stumble into a dead-end, or to move downwards in the tree after having climbed. Every computable strategy for black must eventually do one of these things, and when this occurs, the reader may observe that white has quick a check-mate response.
\end{proof}

Consider now a symmetric version of the previous position, containing a copy of it as well as its dual.
\begin{theorem}\label{Theorem.ComputableTreeWinForFirst}
There is a board position in infinite chess with the following features:
 \begin{enumerate}
  \item The board position is computable.
  \item In the category of computable strategies, the position is a win for the first player (whether this should be white or black).
  \item In the category of all strategies, the position is a draw.
 \end{enumerate}
\end{theorem}

\begin{proof}
The idea of this proof is to combine the previous tree position on half of the board, with a copy of its dual, swapping the roles of white and black, on the other half, as in figure \ref{Figure.SymmetricTrees}.
\begin{figure}[h]
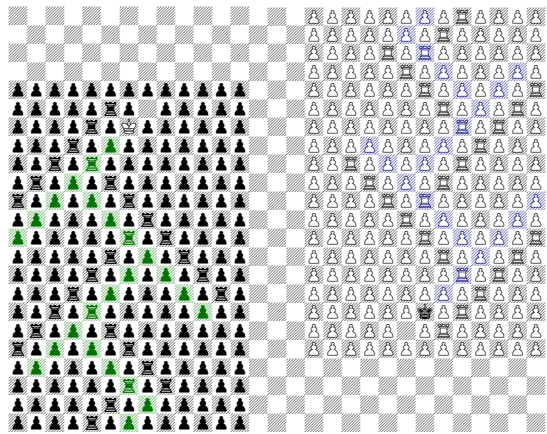

\hfil\chessboard[maxfield=n23,
            boardfontsize=7pt,
            label=false,
            showmover=false,
            borderleft=false,
            borderright=false,
            bordertop=false,
            borderbottom=false,
            marginright=false,
            coloremphstyle=\color{green!45!black},
            emphfields={f16,e15,d14,c13,b12,a11,e13,f12,g11,h10,i9,j8,k7,g9,f8,e7,d6,c5,b4,e5,f4,g3,h2,g1},
            setfen=%
/%
/%
/%
/%
ppppppppppppp/%
ppppprp1ppppp/%
pppprpKpppppp/%
ppprppppppppp/%
pprprpppppppp/%
prppprppppppp/%
rppppprpppppp/%
ppppppprppppp/%
pppppprprpppp/%
ppppprppprppp/%
pppprppppprpp/%
ppprppppppprp/%
pprprpppppppp/%
prppprppppppp/%
rppppprpppppp/%
ppppppprppppp/%
pppppprprpppp/%
ppppprppppppp/%
pppprpppppppp/%
]%
\chessboard[maxfield=o23,
            boardfontsize=7pt,
            label=false,
            showmover=false,
            borderleft=false,
            borderright=false,
            bordertop=false,
            borderbottom=false,
            marginleft=false,
            coloremphstyle=\color{blue!45!black},
            emphfields={j8,k9,l10, m11, n12, o13, p14, q15, r16, k11, j12, i13, h14, g15, f16, i15, j16, k17, l18, m19, n20, k19, j20, i21, h22, i23, j24},
            setfen=%
2PPPPPPPPRPPPPP/%
2PPPPPPPRPPPPPP/%
2PPPPRPRPPPPPPP/%
2PPPPPRPPPPPPPP/%
2PPPPPPRPPPPPRP/%
2PPPPPPPRPPPRPP/%
2PPPPPPPPRPRPPP/%
2PPPPPPPPPRPPPP/%
2PPRPPPPPRPPPPP/%
2PPPRPPPRPPPPPP/%
2PPPPRPRPPPPPPP/%
2PPPPPRPPPPPPPP/%
2PPPPPPRPPPPPRP/%
2PPPPPPPRPPPRPP/%
2PPPPPPPPRPRPPP/%
2PPPPPPPPPRPPPP/%
2PPPPPPkPRPPPPP/%
2PPPPP1PRPPPPPP/%
2PPPPPPPPPPPPPP/%
/%
/%
/%
/%
]\hfil%
\caption{The position is symmetric between white and black.
The first player can computably defeat all computable
opposing strategies}\label{Figure.SymmetricTrees}
\end{figure}
The point now is that both players have (non-computable, but low) strategies to climb with their kings through the corresponding tree, and thus in the category of all strategies, either player can force infinite play and hence a draw. But meanwhile, the first player to move, whether this is white or black, can initiate the pawn-checking strategy so as to force the other player through the tree. This is a computable strategy for the first player, which will defeat any computable strategy for the opponent, since any such strategy will inevitably lead the opposing king into a dead-end of the tree and hence into check-mate.
\end{proof}

The following variation on the theme reveals a more subtle possibility.

\begin{theorem}There is a board position in infinite chess with the following features:
 \begin{enumerate}
  \item The board position is computable.
  \item In the category of computable strategies, the position is undetermined: every computable strategy for one player is defeated by some computable strategy for the other player.
 \item In the category of all strategies, the position is a draw, and both players have drawing strategies of low complexity, which each also defeat any computable strategy for the opponent.
 \end{enumerate}
\end{theorem}

\begin{proof}
By \cite{JockuschSoare1971:AMinimalPairOfPi^0_1Classes}, we may fix two computable binary trees $S$ and $T$, such that no infinite branch through one tree computes a branch for the other. Now, we may design a chess position based on these trees, in a way similar to the position in figure \ref{Figure.SymmetricTrees}, in such a way that black may force white in effect to climb through a copy of $S$, at the same time that white forces black to climb through $T$. The positions should differ from that in figure \ref{Figure.SymmetricTrees}, however, in that the climbing process should not consist entirely of checking moves as it does there, but rather the pushing-up-the-tree process should be periodically interrupted by the need for an extra non-checking move; for example, one can imagine that at certain points, an extra move needs to be completed in order to place the piece that will allow the pushing-up-the-tree process to continue. The effect of such a change will be that immediately after such a non-checking move, control over play would temporarily transfer to the other player, who could begin another round of checking moves forcing their opponent up the other tree, until they reach the point in their own tree where an extra non-checking move is required. In this way, the players in effect take turns forcing their opposing kings to climb the respective trees, with alternating periods of control. Both players have again have drawing strategies of low complexity, since every computable tree has a low branch, and each player can adopt a strategy to climb the tree on a fixed infinite branch through that tree. By following such a strategy, they will never be checkmated, since they will never climb into a dead-end part of the tree. Meanwhile, suppose that white follows a computable strategy $\tau$. Let us imagine that black follows the strategy $\sigma_b$ of climbing the infinite branch $b$ through $T$. Since $b$ does not compute any branch through $S$, it follows that white cannot be following a branch in $T$, and so he will inevitably find himself in a dead-end part of the tree, where he will be checkmated. So $\sigma_b$ for black defeats any computable strategy for white. But further, notice that in order to achieve the checkmate against $\tau$, black climbed only partway up $b$, and so even though $b$ is not computable, there is a computable strategy $\sigma'$ that climbs up $b$ at least that far. And playing $\sigma'$ against $\tau$ will result in the same play of the game as playing $\sigma$ against $\tau$, at least for that length of play, which was long enough to achieve checkmate.\footnote{This is at bottom a general infinitary game-theoretic observation: if any strategy $\sigma$ defeats some strategy $\tau$ in an open game, then there is a computable strategy $\sigma'$ defeating $\tau$, namely, the strategy $\sigma'$ that plays just like $\sigma$ for the fixed finitely many steps required to beat $\tau$.} Thus, we have shown that for every computable strategy $\tau$ for white, there is a computable strategy $\sigma$ for black defeating it. And similarly with the colors reversed.
\end{proof}

We close this section with the observation that it is clear that all of our arguments relativize to oracles. For example, in the case of theorem
\ref{Theorem.ComputableTree}, the same method shows that for any Turing degree $d$, we may use an infinite binary branching tree $T$ that is $d$-computable, but which has no $d$-computable infinite branch, to construct a $d$-computable chess position that is a draw, but in which
white has a computable strategy to defeat any $d$-computable strategy for black. And similarly for the other theorems.

\section{The $\omega_1$ of infinite three-dimensional chess}\label{Section.ThreeD}

Chess in three dimensions has a surprisingly rich history, with numerous variants, some going back well over a century. Kieseritzky's Kubikschack (1851) was played on the $8\times 8\times 8$ board, a format used initially also with Maack's Raumschach (1907), but the Hamburg Raumschach club he founded, which ran from 1919 until World War II, eventually settled on the more manageable $5\times 5\times 5$ board. This game can be played by unpacking the five levels and laying the boards side-by-side, but purists play with no boards at all, holding the position purely in their minds. In several Star Trek episodes, Spock and Kirk play tri-dimesional chess, a futuristic variant with boards of different sizes on different levels, played now by many fans. In his encyclopedia of chess variants, Pritchard \cite{Pritchard2007TheClassifiedEncyclopediaOfChessVariants} describes dozens of variants of three-dimensional and higher-dimensional chess. The second author recalls playing $8\times 8\times 3$ chess as a child.

Here, we consider infinite three-dimensional chess, played in $\Z\times\Z\times\Z$ space, an endless orthogonal lattice of chess cells alternating in color in each of the principal coordinate axis directions. Let us be clear on the movement of the pieces, since there is room here for reasonable people to disagree. Rooks move freely in the three main orthogonal directions: forward/back, left/right, up/down. Bishops move freely on the same-colored diagonals; note that this does not include the long three-dimensional diagonal, which is called the unicorn move. Queens move like rooks or bishops (and so also cannot make the unicorn move). Kings may move in the same directions as a queen, but only one step. White pawns advance in the forward direction along a designated principal axis, with black pawns moving in the opposite direction on this axis. In particular, pawn movement (excepting captures) is one-dimensional, and opposite-color pawns can block each other from advancing.\footnote{This contrasts with the situation in Raumschach, where pawns have two principal directions of movement.} Pawns capture one step on their forward same-color diagonals, that is, to any of four cells. The positions relevant for our proof of theorem \ref{Theorem.3DAnyValueAtttained} will involve only kings, pawns and bishops, and we expect that the basic proof idea will adapt to any of the usual piece movement rules.

We aim in theorem \ref{Theorem.3DAnyValueAtttained} to settle the three-dimensional analogue of conjecture \ref{Conjecture.OmegaOneCH=OmegaOne}, by proving that the omega one of infinite three-dimensional chess $\omegaoneChthreei$ is true $\omega_1$, which by theorem \ref{Theorem.UpperBounds} is as large as it can be.

\begin{theorem}\label{Theorem.3DAnyValueAtttained}
 Every countable ordinal arises as the game value of a position in infinite three-dimensional chess. Hence, $$\omegaoneChthreei=\omega_1.$$
\end{theorem}

We will prove this theorem by embedding into infinite three-dimensional chess another much simpler game, whose game values nevertheless exhaust the countable ordinals. Specifically, to each nonempty tree $T$ on $\omega$, which means that $T\of\omega^\ltomega$ and any initial segment of an element of $T$ is also in $T$, we associate the {\df climbing-through-$T$} game, played as follows.
There are two players, black and white, although the game is completely controlled by black, who aims to climb through the tree $T$ while white watches. Black begins the game by
standing on the root node of $T$, and white responds by saying, ``OK, fine.'' At any stage of the game, black is standing on a node of $T$, and for his move, he selects any immediate successor of that node in $T$, to which white replies, ``OK, fine.'' At a branching node, black has a choice of nodes to move to, while at a leaf node---a terminal node of $T$---there is no successor node for black to choose and in such a case and only in such a case he loses the game. Thus, white wins if black should climb himself into a terminal node of the tree, and black wins only by climbing up an infinite branch through $T$. Since any white win in the climbing-through-$T$ game occurs at a finite stage, if it occurs at all, this game is an open game for white, and so the theory of ordinal game values applies.

It is clear that black can win the climbing-through-$T$ game if and only if there is an infinite branch through $T$, and so white wins the game if and only if the tree $T$ is well-founded. The classical theory of well-founded trees and indeed, of well-founded relations generally, assigns to each node in any well-founded relation an ordinal rank. Specifically, every node is recursively assigned the ordinal supremum of the ordinal successors of the ranks of the successor nodes to it in the tree: $\rank(p)=\sup\set{\rank(q)+1\mid p\to_T q}$, where $p\to_T q$ means that $q$ is an immediate successor of $p$. Thus, leaves have rank $0$; nodes all of whose successors are leaves have rank $1$; the rank of a branching node is the supremum of the ordinal successors of the ranks of the nodes to which it branches. In this article, we shall draw the trees growing upward, although the recursive rank definition proceeds from the leaves downward. In particular, the ordinal ranks of nodes strictly descends as one climbs the tree.
\begin{figure}[h]
\hfil\begin{tikzpicture}[xscale=.3,yscale=.3]
 \draw (-1,1) --(0,0) --(-1,2);
 \draw (-.5,3) --(0,0) --(.5,-2);
 \draw[dotted] (0,2) --(1,2);
 \node at (-5,0) {};
\end{tikzpicture}
\hfil
\begin{tikzpicture}[xscale=.2,yscale=.2]
 \draw (2,-4) --(0,0) --(-4,4) --(-5,5);
 \draw (-5,6) --(-4,4) --(-4.5,7);
 \draw[dotted] (-4,6) --(-3,6);
 \draw (0,0) --(-3,8) --(-4,9);
 \draw (-4,10) --(-3,8) --(-3.5,11);
 \draw[dotted] (-3,10) --(-2,10);
 \draw (0,0) --(-1,12) --(-2,13);
 \draw (-2,14) --(-1,12) --(-1.5,15);
 \draw[dotted] (-1,14) --(0,14);
 \draw (0,0) --(2,16) --(1,17);
 \draw (1,18) --(2,16) --(1.5,19);
 \draw[dotted] (2,18) --(3,18);
 \draw[dotted,thick] (3,12) --(4.9,12);
\end{tikzpicture}
\hfil
\caption{Well-founded trees, with ranks $\omega+2$ and $\omega\cdot 2+3$, respectively}\label{Figure.WellfoundedTrees}
\end{figure}
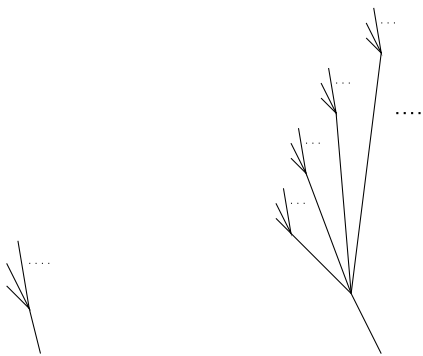
It is an elementary exercise to see that there are well-founded trees $T\of\omega^\ltomega$ of any desired countable ordinal rank. To see this, observe first that if we have a tree $T$ of rank $\alpha$, then the tree $1+T$ obtained by adding an additional step just before the root node of $T$ will have value $\alpha+1$. Second, if we have trees $T_n$ of value $\beta_n$, then the well-founded tree $\oplus T_n$ obtained by having a branching root node, whose successors are the root nodes of the various $T_n$---one may simply prepend the digit $n$ to every sequence in $T_n$ and take the union---will have value $\sup_n \beta_n+1$. It follows by transfinite induction that every countable ordinal is realized as the rank of a well-founded tree.

\begin{lemma}\label{Lemma.ValueOfClimbingThroughT}
 For any nonempty well-founded tree $T\of\omega^\ltomega$, the game value of the climbing-through-$T$ game is equal to the rank of $T$ as a well-founded tree.
\end{lemma}

\begin{proof}
The two ordinal values are defined by the same recursion. The value of a position $p$ in the climbing-through-$T$ game is the same as the rank of that node in $T$.
\end{proof}

In order to prove theorem \ref{Theorem.3DAnyValueAtttained}, we shall prove that for every tree $T$, there is an embedded copy of the climbing-through-$T$ game in infinite three-dimensional chess. What we mean is that there is a position in infinite three-dimensional chess such that the plays of black in this position correspond to moves in the climbing-through-$T$ game. What we will do is to embed an actual copy of the tree $T$, in such a way that white can force black to climb with his king through this copy of $T$, becoming checkmated only when he comes to what corresponds to a terminal node of the tree. We shall argue that the game value of this embedded game is at least as large as the game value of the climbing-through-$T$ game, and therefore any countable ordinal arises as a game value in infinite three-dimensional chess.

We already saw in two dimensions how to implement tree-like structures in infinite chess, for white was already forcing black up the tree in the positions of figures \ref{Figure.TreeChannels}, \ref{Figure.ForcingBlackToClimb} and \ref{Figure.SymmetricTrees}. The crucial difference there, however, was that those trees were merely 2-branching, since the branching nodes had only two successor nodes. But since any finitely-branching well-founded tree is finite, such trees will have finite rank and therefore will not help us with conjecture \ref{Conjecture.OmegaOneCH=OmegaOne}. In order to achieve high countable ordinal ranks, we must consider trees that have branching nodes with infinitely many immediate successors. The main difficulty with implementing such a branching node in infinite chess is that we must design the position so that white can force black to make a definite choice, without allowing black to cause a draw just by extending indefinitely the play implementing the choice itself. That is, it won't be enough to have a tree-like structure where black is able to choose any of the infinitely many branching nodes, if black is also able in effect by delay not to make any choice at all, since then black will choose to avoid making the choice and thereby cause a draw. We have already seen this feature in the $\omega^3$ position of figure \ref{Figure.ValueOmegaCubed}, where we gave black the opportunity to activate any of the black rook towers, but he definitely had to choose one of them, by moving his black bishop up the diagonal as far as he liked. We shall implement a similar construction here. The complication is that the tree $T$ will have infinitely many such branching nodes, and implementing a single branching node already takes up a whole diagonal. In two dimensions, there simply doesn't appear to be enough room to embed well-founded trees with appreciable rank in this way. So although we had originally wanted to use this well-founded tree idea in ``normal'' (that is, two-dimensional) infinite chess, we found ourselves up against the chess geometry difficulty that the idea simply didn't seem to fit in two dimensions. Meanwhile, the idea does work in three dimensions, as we shall presently explain, and so we are happy to settle the question for infinite three-dimensional chess.

The main idea is to give each branching node of the well-founded tree $T$ its own layer in the three-dimensional chess playing area. Each such branching layer will involve an infinite diagonal channel, as in figure \ref{Figure.3DBranchingLayer}, with as many side channels branching off of it as the node as successors in $T$, infinitely many if desired. The details of the position will enable white to force black to choose one of these side channels or be checkmated, thereby simulating the choice black faces in the climbing-through-$T$ game. Ultimately, the position will resemble a series of tubes in space, as pictured abstractly in figure \ref{Figure.EmbeddedTree}, and the main line of play will call for white to force black through these tubes, whose connective structure is abstractly the same as in the tree $T$, except that each branching node of $T$ is simulated on a separate layer of the position.
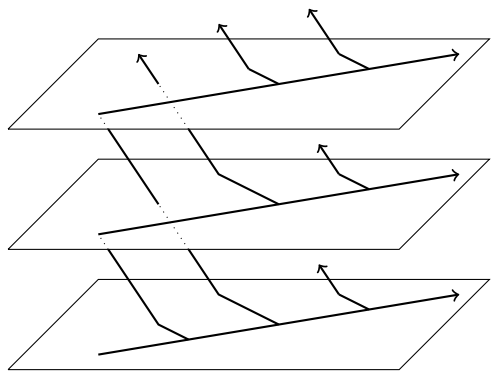
\begin{figure}[h]
\hfil\begin{tikzpicture}[xscale=.4,yscale=.2]
 \draw (0,0) -- (13,0) -- (16,6) -- (3,6) -- (0,0);
 \draw (0,8) -- (13,8) -- (16,14) -- (3,14) -- (0,8);
 \draw (0,16) -- (13,16) -- (16,22) -- (3,22) -- (0,16);
 \draw[->,thick] (3,1) -- (15,5);
 \draw[->,thick] (3,9) -- (15,13);
 \draw[->,thick] (3,17) -- (15,21);
 \draw[thick] (6,2) -- (5,3) -- (3.33,8);
 \draw[dotted] (3.33,8) -- (3,9);
 \draw[thick] (9,3) -- (7,5) -- (6,8);
 \draw[dotted] (6,8) -- (5,11);
 \draw[thick] (5,11) -- (3.33,16);
 \draw[dotted] (3.33,16) -- (3,17);
 \draw[thick] (12,4) -- (11,5);
 \draw[->,thick] (11,5) -- (10.33,7);
 \draw[thick] (9,11) -- (7,13) -- (6,16);
 \draw[dotted] (6,16) -- (5,19);
 \draw[->,thick] (5,19) -- (4.33,21);
 \draw[thick] (12,12) -- (11,13);
 \draw[->,thick] (11,13) -- (10.33,15);
 \draw[thick] (9,19) -- (8,20);
 \draw[->,thick] (8,20) -- (7,23);
 \draw[thick] (12,20) -- (11,21);
 \draw[->,thick] (11,21) -- (10, 24);
\end{tikzpicture}
\hfil
\caption{The embedded tree, with the infinite-branching nodes of $T$ simulated individually on separate layers}\label{Figure.EmbeddedTree}
\end{figure}
Specifically, if the black king should enter this branching layer, he will find himself at the start of an infinite diagonal as in figure \ref{Figure.3DBranchingLayer} with an obstructing black bishop. Black will have one spare move here to move the black bishop out of the way or be checkmated, and he can move it as far as he desires up the diagonal, which will enable access to any of the side channels branching off before that point. The idea, then, is that white will proceed to check the black king, forcing black to climb the channel and choose a branching path, or else be checkmated against the obstructing bishop. Once black selects one of the branching channels, white will continue checking him into this channel. Each such off-shoot channel either leads to a dead-end (corresponding to a leaf node in $T$), in which the black king is check-mated by the pawn as he lands at the end of the channel, or else leads to a ``staircase'' channel taking the black king onto another layer higher up, as shown in figure \ref{Figure.3dTransitionToStairwell}. Thus, upon entering a branching layer corresponding to a node in $T$, white forces black to choose one of the side channels corresponding to one of the successor nodes in $T$.

In the main position, black will have a mate-in-two position against white, with nothing white can do about it nor any need for white to reply, in a protected part of the board that does not interact with the embedded tree position we are about to describe. Meanwhile, in the main line of play, every white move will either check the black king or threaten mate-in-one, and consequently, black will never be at liberty to make even the first of his mate-in-two moves. So although the main line of play will be very long, in a way controlled by black so as to witness very high transfinite game values, nevertheless both players will in a sense be very close to checkmate during the entire play. This tense nature of the game will make it easier for us to argue that play must follow the main line, for neither player is at liberty for any spare moves. For example, although the position is somewhat open for white, in the sense that white could make many moves that tend to destroy the tree structure in various complicated ways, the fact is that none of these moves will check black or threaten forced-mate-with-checks (unless one has already reached a terminal part of the tree), and so white simply doesn't have time for them. Any move by white that doesn't check black or threaten forced-mate-with-checks will be answered by progress in black's mate-in-two region of the board, allowing black to win on the next move.

\begin{figure}[h]
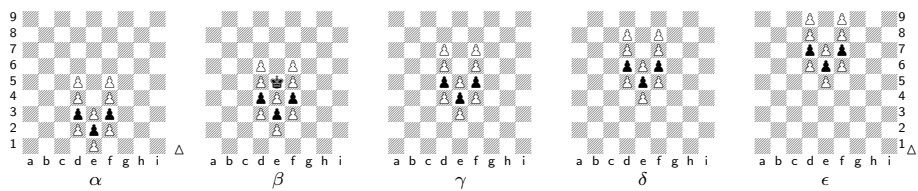
\hskip0pt
\hfil\vtop{\hbox{\chessboard[maxfield=i9,
            boardfontsize=6pt,
            label=false,
            labelleft=true,
            labelbottom=true,
            labelfontsize=4pt,
            labelleftwidth=1.6ex,
            border=false,
            margin=false,
            arrow=to,
            backmove={e3-e4},
            setfen=%
/%
/%
/%
/%
3P1P/%
3P1P/%
3pPp/%
3PpP/%
4P/%
]}\hbox{\hskip25pt$\scriptstyle\alpha$}}%
\hfil
\vtop{\hbox{\chessboard[maxfield=i10,
            printarea = a2-i10,
            boardfontsize=6pt,
            label=false,
            labelbottom=true,
            labelfontsize=4pt,
            labelleftwidth=1.6ex,
            showmover=false,
            border=false,
            margin=false,
            setfen=%
/%
/%
/%
3P1P/%
3PkP/%
3pPp/%
3PpP/%
4P/%
]}\hbox{\hskip25pt$\scriptstyle\beta$}}%
\hfil
\vtop{\hbox{\chessboard[maxfield=i9,
            boardfontsize=6pt,
            label=false,
            labelbottom=true,
            labelfontsize=4pt,
            labelleftwidth=1.6ex,
            showmover=false,
            border=false,
            margin=false,
            setfen=%
/%
/%
3P1P/%
3P1P/%
3pPp/%
3PpP/%
4P/%
/%
]}\hbox{\hskip25pt$\scriptstyle\gamma$}}%
\hfil
\vtop{\hbox{\chessboard[maxfield=i10,
            printarea = a2-i10,
            boardfontsize=6pt,
            label=false,
            labelbottom=true,
            labelfontsize=4pt,
            labelleftwidth=1.6ex,
            showmover=false,
            border=false,
            margin=false,
            setfen=%
/%
3P1P/%
3P1P/%
3pPp/%
3PpP/%
4P/%
/%
/%
]}\hbox{\hskip25pt$\scriptstyle\delta$}}%
\hfil
\vtop{\hbox{\chessboard[maxfield=i9,
            boardfontsize=6pt,
            label=false,
            labelright=true,
            labelbottom=true,
            labelfontsize=4pt,
            labelleftwidth=1.6ex,
            mover=w,
            border=false,
            margin=false,
            setfen=%
3P1P/%
3P1P/%
3pPp/%
3PpP/%
4P/%
/%
/%
/%
]}\hbox{\hskip25pt$\scriptstyle\epsilon$}}\hfil%
\caption{The black king is forced to ascend the stairs via {\tt 1.$\alpha$e4+ K$\gamma$e6 2.$\beta$e5+ K$\delta$e7 3.$\gamma$e6+ K$\epsilon$e8 4.$\delta$e7+}.}\label{Figure.3dStairway}
\end{figure}
As a warm-up to our position and three-dimensional chess generally, consider the basic stairway position, shown in figure \ref{Figure.3dStairway}, in which white forces the black king to ascend. Imagine the layers stacked atop each other, with $\alpha$ at the bottom and further layers below and above. The black king had entered at {\tt $\alpha$e4}, was checked from below and has just moved to {\tt $\beta$e5}. Pushing a pawn with check, white continues with {\tt 1.$\alpha$e4+ K$\gamma$e6 2.$\beta$e5+ K$\delta$e7 3.$\gamma$e6+ K$\epsilon$e8 4.$\delta$e7+}, forcing black to climb the stairs (the pawn advance {\tt 1.$\alpha$e4+} was protected by a corresponding pawn below, since black had just been checked at {\tt $\alpha$e4}). We encourage the reader to verify as an exercise that each black move here is forced (for example, {\tt $\beta$f6} is protected by the pawn at {\tt $\alpha$f5}). Furthermore, white has no advantage to moving any of the white pawns early, as this merely opens a gap in the stairwell protection, allowing black to escape, or creates a refuge where black can rest without fear of check, during which time he would implement his mate-in-$2$ position against white elsewhere.

This stairway pattern can continue indefinitely, but in our position the stairways lead eventually to what we call a branching layer, as in figure \ref{Figure.3DBranchingLayer}, showing the main branching layer $\gamma$, with the transition from the stairway indicated in levels $\omega$, $\alpha$ and $\beta$ below it. The black king will enter from the stairway below at {\tt $\omega$e3} and is then pushed by a pawn check from below to {\tt 1\ldots K$\alpha$e4 2.$\omega$e3+ K$\beta$e5 3.$\alpha$e4+ K$\gamma$e6}. If white could now check the black king, it would be checkmate, since the layer above stops the stairway ascent and the black bishop blocks the diagonal in layer $\gamma$, so black has nowhere to move his king.
\begin{figure}[h]
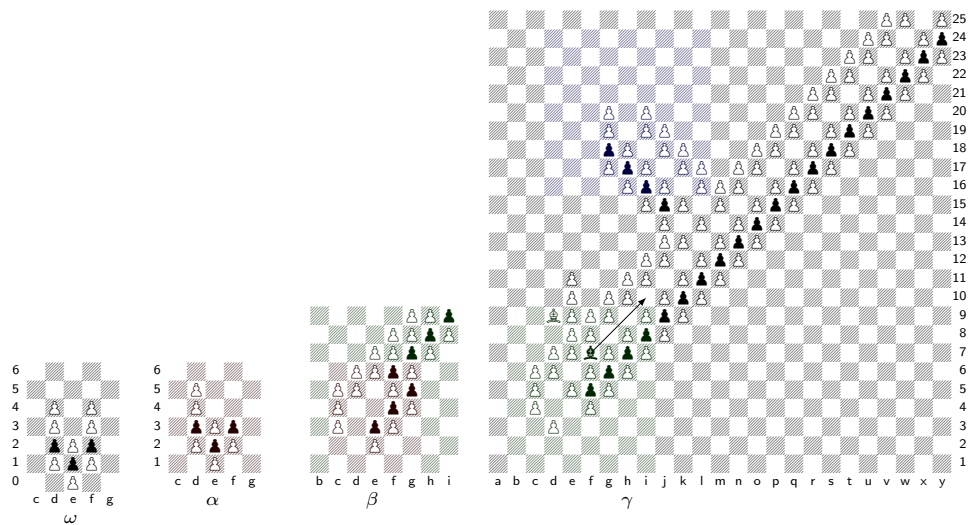

\hskip0pt
\lower 7pt\hbox{\vtop{\hbox{\chessboard[maxfield=g6,
            printarea=c0-g6,
            zero,
            boardfontsize=7pt,
            label=false,
            labelbottom=true,
            labelleft=true,
            labelfontsize=4pt,
            labelleftwidth=1.6ex,
            showmover=false,
            border=false,
            setfen=%
/%
/%
4P1P/%
4P1P/%
4pPp/%
4PpP/%
5P/%
]}%
\hbox{\hskip21pt$\scriptstyle\omega$}}}%
\hfil
\vtop{\hbox{\chessboard[maxfield=g7,
            printarea=c1-g6,
            boardfontsize=7pt,
            label=false,
            labelbottom=true,
            labelleft=true,
            labelfontsize=4pt,
            labelleftwidth=1.6ex,
            showmover=false,
            border=false,
            fieldcolor=red!15!black,
            setfontcolors,
            setfen=%
/%
/%
3P/%
3P/%
3pPp/%
3PpP/%
4P/%
]}%
\hbox{\hskip21pt$\scriptstyle\alpha$}}%
\hfil
\vtop{\hbox{\chessboard[maxfield=i8,
            printarea=b0-i8,
            zero,
            boardfontsize=7pt,
            label=false,
            labelbottom=true,
            labelfontsize=4pt,
            labelleftwidth=1.6ex,
            showmover=false,
            border=false,
            coloremphstyle=\color{green!15!black},
            empharea=a0-i8,
            coloremphstyle=\color{red!15!black},
            empharea=c0-g5,
            setfen=%
7PPpP/%
6PPpP/%
5PPpP/%
4PPpP/%
3PP1Pp/%
3P2pP/%
3P1pP/%
5P/%
]}%
\hbox{\hskip 28pt$\scriptstyle\beta$}}%
\hfil
\vtop{\hbox{\chessboard[maxfield=y25,
            boardfontsize=7pt,
            label=false,
            labelbottom=true,
            labelfontsize=4pt,
            labelrightwidth=1.6ex,
            labelright=true,
            showmover=false,
            margin=false,
            border=false,
            coloremphstyle=\color{green!15!black},
            empharea=b1-i9,
            coloremphstyle=\color{blue!25!black},
            empharea=d16-l24,
            pgfstyle=straightmove,
            linewidth=.05ex,
            backmove=f7-i10,
            setfen=%
21PP1PpP/%
20PP1PpP/%
19PP1PpP/%
18PP1PpP/%
17PP1PpP/%
6P1P7PP1PpP/%
6P1PP5PP1PpP/%
6pP1PP3PP1PpP/%
6PpP1PP1PP1PpP/%
7PpP1PPP1PpP/%
8PpP1P1PpP/%
9P1P1PpP/%
9PP1PpP/%
8PP1PpP/%
4P2PP1PpP/%
4P1PP1PpP/%
3BPPP1PpP/%
4PP1PpP/%
3PPbPpP/%
2PP1PpP/%
2P1PpP/%
2P2P/%
3P]}%
\hbox{\hskip50pt$\scriptstyle\gamma$}}
\caption{Transition from stairwell to branching layer $\gamma$; additional pawns in layer above and below will confine black king to the channel}\label{Figure.3DBranchingLayer}%
\label{Figure.3DBranchingLayer}
\end{figure}
But white has no immediate checking move. Instead, white sets up the check with {\tt 4.$\gamma$d4}, threatening mate with {\tt $\gamma$d5\#}, unless black moves the bishop out of the way. Note that black should not capture the white pawn at {\tt $\gamma$e8}, since then white will capture with his bishop, leading again to checkmate {\tt $\gamma$d5\#}. We assume similarly that additional white bishops are placed just behind the white pawns in the layers above and below (and those white bishops should similarly be surrounded by white pawns preventing them from causing any immediate damage to the tree structure). Thus, the only direction for black to move his bishop out of the way is up the free diagonal in layer $\gamma$. The main point---indeed, the main idea of this entire position---consists of the observation now that black will move his bishop from {\tt$\gamma$f7} up some distance in this diagonal channel, past finitely many of the branching-off passages of the sort of which only one is shown in this figure, thereby giving black the chance eventually to go into any of them that he wishes, while also being forced to choose at least one. White will now continue in layer $\gamma$ with checking moves {\tt 5.d5+ Kf7 6.e6+ Kg8 7.f7+ Kh9 8.g8+ Ki10}, and so on. Eventually, black reaches one of the branching-off channels, such as at {\tt 12.k12+ Km14 13.l13+}, where he faces a free branching choice. On the one hand, he may simply ignore the branch channel and continue up the main diagonal with {\tt 13\ldots Kn15 14.m14+ Ko16} and so on. Eventually, however, he will be check-mated against his bishop, which blocks the channel higher up, unless he opts for one of the branching channels (and this is the key step, where we have effectively forced black to choose one of the infinitely many branching channels available). To illustrate with choosing the pictured channel, consider {\tt 13\ldots Kl15}. If white advances the {\tt m} pawn, then it could be captured by the black bishop coming back down, which could then be captured by the white pawn now at {\tt l13}, but then since this pawn is not protected, black could capture it safely with his king and find a momentary haven, sufficient to win with his mate-in-two position elsewhere. So instead white checks with {\tt 14.k14+}. If black moves his king, play continues with {\tt 14\ldots Kk16 15.m15+ Kj17 16.k16+} and so on up the channel. Alternatively, if black captures the pawn with {\tt 14\ldots xk14}, then play continues with {\tt 15.xk14+}, which is protected from a pawn in upper layer, and then {\tt 15\ldots Kk16 16.m15+ Kj17} and so on again up the channel. Thus, black is able in effect to choose any of the branching channels before his bishop, and is thereby able when making the original bishop move for this layer to have decided on any of the infinitely many branching channels, and simply move his bishop far enough out so that he is able when reaching that channel to opt for it. In this way, black can choose any desired side channel.

Let us now flesh out how the position transforms from the branching layer $\gamma$ to another stairway, which may lead to another branching layer or to a dead-end in which black is checkmated. The left board $\gamma$ here shows the corresponding part of layer $\gamma$ in figure \ref{Figure.3DBranchingLayer} at the end of the side branch channel, after the black king has opted for that side channel.
\begin{figure}[h]
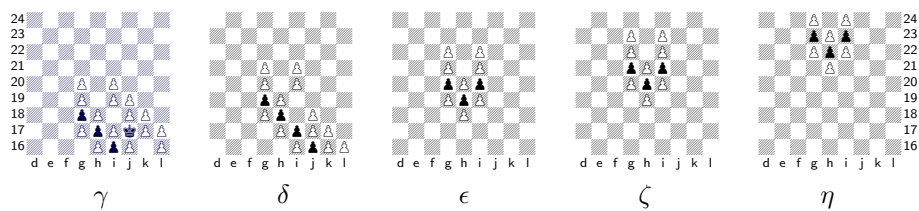

\hfil\chessboard[maxfield=l24,
            printarea = d16-l24,
            startfen=d24,
            boardfontsize=6pt,
            label=false,
            labelleft=true,
            labelbottom=true,
            labelfontsize=4pt,
            labelleftwidth=2ex,
            showmover=false,
            border=false,
            margin=false,
            arrow=to,
            backmove={e3-e4},
            coloremphstyle=\color{blue!25!black},
            empharea=d16-l25,
            setfen=%
/%
/%
/%
/%
3P1P/%
3P1PP/%
3pP1PP/%
3PpPkPP/%
4PpP1PP/%
]%
\hfil
\chessboard[maxfield=l25,
            printarea = d17-l25,
            startfen=d25,
            boardfontsize=6pt,
            label=false,
            labelbottom=true,
            labelfontsize=4pt,
            labelleftwidth=1.6ex,
            showmover=false,
            border=false,
            margin=false,
            setfen=%
/%
/%
/%
3P1P/%
3P1P/%
3pP/%
3Pp1P/%
4PpPP/%
5PpPP
]%
\hfil
\chessboard[maxfield=l24,
            printarea = d16-l24,
            startfen=d24,
            boardfontsize=6pt,
            label=false,
            labelbottom=true,
            labelfontsize=4pt,
            labelleftwidth=1.6ex,
            showmover=false,
            border=false,
            margin=false,
            setfen=%
/%
/%
3P1P/%
3P1P/%
3pPp/%
3PpP/%
4P/%
/%
]%
\hfil
\chessboard[maxfield=l25,
            printarea = d17-l25,
            startfen=d25,
            boardfontsize=6pt,
            label=false,
            labelbottom=true,
            labelfontsize=4pt,
            labelleftwidth=1.6ex,
            showmover=false,
            border=false,
            margin=false,
            setfen=%
/%
3P1P/%
3P1P/%
3pPp/%
3PpP/%
4P/%
/%
/%
]%
\hfil
\chessboard[maxfield=l24,
            printarea = d16-l24,
            startfen=d25,
            boardfontsize=6pt,
            label=false,
            labelbottom=true,
            labelright=true,
            labelfontsize=4pt,
            labelrightwidth=2ex,
            showmover=false,
            margin=false,
            border=false,
            setfen=%
3P1P/%
3P1P/%
3pPp/%
3PpP/%
4P/%
/%
/%
/%
]\hfil%

\medskip

\hfill$\gamma$\ \ \ \ \ \hfill\ \ $\delta$\ \ \ \ \ \hfill\ \ $\epsilon$\ \ \ \ \ \hfill\ \ $\zeta$\ \ \ \ \ \hfill\ \ $\eta$\hfill\hskip0pt

\caption{The stairwell bottom: king enters from horizontal layer on bottom level $\alpha$ and will be forced by pawn checks to ascend the staircase $\tt \gamma h19\to\delta h20\to\epsilon h21\to\zeta h22\to\eta h23$, and so on.}\label{Figure.3dTransitionToStairwell}
\end{figure}
The lower right portion of board $\delta$ here shows part of what appears in the layer immediately above layer $\gamma$ in the main branching layer. White forces the king at {\tt$\alpha$j17} to enter the stairwell and ascend the stairs with: {\tt 1.$\gamma$k16+ K$\gamma$i18 2.$\gamma$j17+ K$\gamma$h19 3.$\gamma$i18+ K$\delta$h20 4.$\gamma$h19+ K$\epsilon$h21 5.$\delta$h20+ K$\zeta$h22 6.$\epsilon$h21+ K$\eta$h23 7.$\zeta$h22+} and so on. The point now is that this stairway can lead to another branching layer as in figure \ref{Figure.3DBranchingLayer}, with the whole assembly following the abstract branching structure of $T$ in the manner shown in figure \ref{Figure.EmbeddedTree}.

We take ourselves now to have sketched an explanation of the various modules that can be used to build up the overall position in infinite three-dimensional chess corresponding to a fixed tree $T$ on $\omega$. Each branching node $u$ of $T$ corresponds to its own branching layer in the manner of figure \ref{Figure.3DBranchingLayer}, whose side channels will correspond to the successors of $u$ in $T$, which may lead via the stairway configurations to branching nodes on higher layers, if those successor nodes lead to further branching nodes in $T$, or to dead-ends in which black can be checkmated, if those successor nodes are leaves in $T$.

Although we admit that the detailed accounts of complex positions in infinite chess can be finicky, let us mention the essential features we believe that our position exhibits. For a fixed tree $T$ on $\omega$, we have described a position $p_T$ in infinite three-dimensional chess, in which white will be able to force the black king through a series of channels corresponding to the nodes of $T$. In particular, this correspondence is tight enough that any strategy for black in the climbing-through-$T$ game can be implemented as a strategy for black in the infinite chess position starting from position $p_T$. As white forces the black king through the channels, it is as though black is standing on the corresponding node of $T$. When this is a branching node of $T$, then the black king is in the corresponding branching layer of $p_T$, which allows black to send his king into any of the side-branching channels, which correspond to the successor nodes in $T$ of that branching node.

It follows that if the tree $T$ is well-founded, then it has no infinite branches, and so ultimately the black king must find itself in a dead-end channel. So white has a winning strategy to win from the chess position $p_T$. Furthermore, the game value of this position $p_T$ must be at least as great as the game value of the climbing-through-$T$ game, since as we have explained, any strategy in the climbing-through-$T$ game can be simulated in black's play from position $p_T$, and consequently the game value of any node in $T$ is bounded by the game value of the position arising from corresponding play from $p_T$. Basically, black will be able to continue playing in the chess position provided that he still has moves to play in the corresponding climbing-through-tree position. Note that the game value of the chess position will generally be somewhat larger than the corresponding position in the tree $T$, since simulating one move in climbing-through-$T$ generally takes many moves in infinite chess, particularly at the branching nodes. Since lemma \ref{Lemma.ValueOfClimbingThroughT} shows that the game values of the climbing-through-$T$ games for various $T$ are unbounded in the countable ordinals, it follows that the game values of the positions in infinite three-dimensional chess must also be unbounded in the countable ordinals, and so the omega one of infinite three-dimensional chess is as large as it can be $$\omegaoneChthreei=\omega_1.$$ This completes the proof of theorem \ref{Theorem.3DAnyValueAtttained}.

We would like to make a few final observations about the proof of theorem \ref{Theorem.3DAnyValueAtttained}. First, observe that by considering only the case of computable well-founded trees $T$, what we attain is a computable position $p_T$ in infinite three-dimensional chess, whose value is at least as great as the rank of $T$. Since these values are unbounded in $\omega_1^{ck}$, and furthermore the value of any computable position with a value is bounded by $\omega_1^{ck}$, it follows that the omega one of computable infinite positions in infinite three-dimensional chess is as large as it can be.

\begin{corollary}\label{Corollary.OmegaOneCh3c}
 Every computable ordinal arises as the game value of a computable position of infinite three-dimensional chess. Consequently,
 $$\omegaoneChthreec=\omega_1^{ck}.$$
\end{corollary}

Second, we note that a variety of interesting positions are obtained by using {\it two} tree simulations, with opposite colors and proceeding in opposite directions. That is, rather than giving black a mate-in-two position elsewhere, we give black his own embedded tree, which is like mate-in-$2$ in the weak sense we have observed that every move is either checking or threatening mate-in-one. So the main line of play for this combination-of-two-trees position will have black and white alternately seizing control, making long series of checking moves until they come to the branch layer, at which time the opposing player will seize control for a time. By comparing the ranks of the two trees, one will realize interesting game values this way.

Finally, we observe that in our proof of theorem \ref{Theorem.3DAnyValueAtttained}, it appears that we don't really need the full scope of three dimensions, but rather, it would be enough to play on a board of type $\mathbb{Z}\times\mathbb{Z}\times m$, where the third dimension has only sufficiently large finite diameter $m$. The idea would be to use bridges that go up and come back down, rather than stairways that go only up, and to place all the branching diagonals in one layer, parallel to each other but spaced apart. When following a side branch, now, the black king will be forced to ascend a bridge that carries him over to the beginning of another branching diagonal, descending again to the original layer, while skipping over the intervening branching diagonals.


\end{document}